\documentclass[11pt,a4paper]{article}

\usepackage{booktabs}
\usepackage{amssymb,latexsym}
\usepackage{a4}
\usepackage[all]{xy}
\usepackage{color}

\usepackage[normalem]{ulem}

\newif\iflabel
\labelfalse
\newcommand{\Label}[1]{\iflabel\ifmmode\makebox[0pt][l]{[#1]}
                       \else\marginpar{[#1]}
                       \fi\fi\label{#1} }

\newcommand{\U}{{\mathcal U}}

\newcommand{\UK}{U\!K}

\newcommand{\bs}{\backslash}
\newcommand{\ra}{\rangle}
\newcommand{\la}{\langle}
\newcommand{\id}{\mathrm{id}}
\newcommand{\X}{{\mathcal X}}
\newcommand{\Q}{{\mathcal Q}}
\newcommand{\R}{{\mathcal R}}
\newcommand{\El}{{\mathcal L}}

\renewcommand{\S}{{\mathcal S}}
\newcommand{\Z}{{\mathbb Z}}
\newcommand{\N}{{\mathbb N}}
\newcommand{\rk}{\mathrm{rk}}
\newcommand{\Aut}{{\mathrm{Aut}}}
\newcommand{\End}{{\mathrm{End}}}
\newcommand{\Stab}{{\mathrm{Stab}}}

\newcommand{\mc}{\mathcal}
\newcommand{\im}{{\mathrm{im}}}

\newcommand{\Gap}{{\scshape Gap}}

\newcommand{\Grig}{{\mathfrak G}}

\newcommand{\ti}{\tilde}

\newcommand{\Sym}{{\mathrm{Sym}}}
\newcommand{\Hom}{{\mathrm{Hom}}}

\newcommand{\Core}{{\mathrm{Core}}}

\usepackage{bbm}

\usepackage{rotfloat}
\floatstyle{plain}
\floatname{algorithm}{Algorithm}
\newfloat{algorithm}{tbp}{lop}

\usepackage{lineno}\linenumbers

\title{{A Reidemeister-Schreier theorem for finitely $L$-presented groups}}
\author{Ren\'e Hartung}
\date{$Date: 2011-06-06 14:56:22 $}

\newenvironment{proof}{\par\vskip-\lastskip\vskip\topsep
\noindent{\it Proof.}\vadjust{\nobreak}\quad
\begingroup\divide\topsep3\divide\itemsep3
\divide\partopsep3\divide\parskip3
\divide\parsep3}
{\ifvmode\penalty10000\hbox to\hsize{\hfil$\Box$}
\else\parfillskip0pt\widowpenalty10000\hfil$\Box$
\fi\par\vskip 1.5ex\endgroup}

\newenvironment{CenProof}[1]{\par\vskip-\lastskip\vskip\topsep
\noindent{\it Proof#1.}\vadjust{\nobreak}\quad
\begingroup\divide\topsep3\divide\itemsep3
\divide\partopsep3\divide\parskip3
\divide\parsep3}
{\ifvmode\penalty10000\hbox to\hsize{\hfil$\Box$}
\else\parfillskip0pt\widowpenalty10000\hfil$\Box$
\fi\par\vskip 1.5ex\endgroup}

\newtheorem{theorem}{Theorem}[section]
\newtheorem{corollary}[theorem]{Corollary}
\newtheorem{lemma}[theorem]{Lemma}
\newtheorem{example}[theorem]{Example}
\newtheorem{conjecture}{Conjecture}
\newtheorem{proposition}[theorem]{Proposition}
\newtheorem{question}{Question}
\newtheorem{remark}[theorem]{Remark}
\newtheorem{definition}[theorem]{Definition}

\begin{document}
\maketitle
\begin{abstract}
  We prove a variant of the well-known Reidemeister-Schreier theorem 
  for finitely $L$-presented groups. More precisely, we prove 
  that each finite index subgroup of a finitely $L$-presented group is
  itself finitely $L$-presented. Our proof is constructive and it yields
  a finite $L$-presentation for the subgroup. We further study conditions
  on a finite index subgroup of an invariantly finitely $L$-presented
  group to be invariantly $L$-presented itself.\bigskip

  \noindent{\it Keywords:} Reidemeister-Schreier theorem; infinite presentations;
  recursive presentations; self-similar groups; Basilica group; Grigorchuk group;
  finite index subgroups;
\end{abstract}

\section{Introduction} 
Group presentations play an important role in computational group theory.
In particular finite group presentations have been subject to extensive
research in computational group theory dating back to the early days of
computer-algebra-systems~\cite{Neu82}. Group presentations, on the one
hand, provide an effective description of the group. On the other
hand, a description of a group by its generators and relations leads
to various decision problems which are known to be unsolvable
in general. For instance, the word problem of a finitely presented
group is unsolvable~\cite{Nov55,Boo59}; see also \cite{LS77}. Though
various total and partial algorithms for finitely presented groups
are known~\cite{Sim94}.  For instance, the coset-enumeration process
introduced by Todd and Coxeter~\cite{TC36} enumerates the cosets of
a subgroup in a finitely presented group. If the subgroup has finite
index, coset-enumeration terminates and it computes a permutation
representation for the group's action on the cosets. Coset-enumeration
is a partial algorithm 
as the process will not terminate if the subgroup has infinite index. 
However, finite presentations often allow total algorithms that will compute factor
groups with a special type (including
abelian quotients, nilpotent quotients~\cite{Nic96} and, in general,
solvable quotients~\cite{Lo98}).\smallskip

Beside quotient and subgroup methods, the well-known theorem
by Reidemeister~\cite{Reid26} and Schreier~\cite{Schr27} allows to
compute a presentation for a subgroup. The Reidemeister-Schreier theorem explicitly shows that
a finite index subgroup of a finitely presented group is itself
finitely presented. A similar result can be shown for finite index
ideals in finitely presented semi-groups~\cite{CRRT95}. In practice,
the permutation representation for the group's action on the cosets allows
to compute the Schreier generators of the subgroup and the Reidemeister
rewriting. The Reidemeister rewriting allows us to rewrite the relations of the group to relations
of the subgroup~\cite{Hav74,Sim94,LS77}. Note that a method to compute a finite presentation
for a finite index subgroup can be applied in the investigation of the structure of
a group by its finite index subgroups; see~\cite{HR94}.\smallskip

Even though finitely presented groups have been studied for a
long time, most groups are not finitely presented as there are
uncountably many two-generator groups~\cite{BHN37} but only countably many finite
presentations~\cite{Bar03}. A generalization of
finite presentations are \emph{finite $L$-presentations} which were
introduced in~\cite{Bar03}; however, there are still only countably many finite $L$-presentations. It is known that various examples
of self-similar or branch groups (including the Grigorchuk group~\cite{Gri80} and
its twisted twin~\cite{BS10}) are finitely $L$-presented but
not finitely presented~\cite{Bar03}. Finite $L$-presentations are
possibly infinite presentations with finitely many generators and
whose relations (up to finitely many exceptions) are obtained by
iteratively applying finitely many substitutions to a finite set of
relations; see~\cite{Bar03} or Section~\ref{sec:Prel} below. A finite
$L$-presentation is \emph{invariant} if the substitutions which generate
the relations induce endomorphisms of the group; see also Section~\ref{sec:Prel}. In fact, 
invariant finite $L$-presentations are finite presentations in the universe 
of \emph{groups with operators} defined in~\cite{Kr25,Noe29} in the sense that
the operator domain of the group generates the infinitely many relations 
out of a finite set of relations.\smallskip

Finite $L$-presentations allow computer algorithms to be applied
in the investigation of the groups they define. For instance, they
allow to compute the lower central series quotients~\cite{BEH08},
the Dwyer quotients of the group's Schur multiplier~\cite{Har10a}, and
even a coset-enumeration process exists for finitely
$L$-presented groups~\cite{Har10b}.  It is the aim of this paper to
prove the following variant of Reidemeister-Schreier's theorem:
\begin{theorem}\Label{thm:CentralThm}
  Each finite index subgroup of a finitely 
  $L$-presented group is finitely $L$-presented. 
\end{theorem} 
If the finite index subgroup in Theorem~\ref{thm:CentralThm} is normal and
invariant under the substitutions (i.e., a normal and admissible subgroup in the 
notion of Krull \& Noether~\cite{Kr25,Noe29}), an easy argument gives a finite
L-presentation for the subgroup; furthermore, if the group is invariantly
finitely L-presented, then so is the subgroup.  However, more work is
needed if the subgroup is not invariant under the substitutions.  Under
either of two extra conditions (the subgroup is \emph{leaf-invariant},
see Definition~\ref{def:LeafInv}; or it is normal and \emph{weakly
leaf-invariant}, see Definition~\ref{def:WeakInv}), we show that the
subgroup is invariantly finitely L-presented as soon as the group is. We
have not been able to get rid of these extra assumptions. In particular,
it is not clear whether a finite index subgroup of an invariantly finitely
$L$-presented group is always invariantly finitely $L$-presented.
We show that the methods presented in this paper will (in general)
fail to compute invariant $L$-presentations for the subgroup even if
the group is invariantly $L$-presented. However, we are not aware of
a method to prove that a given subgroup does not admit an invariant
finite $L$-presentation at all.\smallskip

Our proof of Theorem~\ref{thm:CentralThm} is constructive and it yields a
finite $L$-presentation for the subgroup. These finite $L$-presentations
can be applied in the investigation of the underlying groups as the methods
in~\cite{HR94} suggest for finitely presented groups.  Notice that
Theorem~\ref{thm:CentralThm} was already posed in Proposition~2.9
of~\cite{Bar03}. The proof we explain in this paper follows the sketch
given in~\cite{Bar03}, but fixes a gap as the $L$-presentation of
the group in Theorem~\ref{thm:CentralThm} is possibly non-invariant. Even if the $L$-presentation is
assumed to be invariant, the considered subgroup cannot be assumed to
be invariant under the substitutions.\smallskip

This paper is organized as follows: In Section~\ref{sec:Prel} we
recall the notion of a finite $L$-presentation and we recall basic
group theoretic constructions which preserve the property of being
finitely $L$-presented.  Then, in Section~\ref{sec:ReidSchr},
we recall the well-known Reidemeister-Schreier
process. Before we prove Theorem~\ref{thm:CentralThm} in Section~\ref{sec:ReidSchrThm}, we construct 
in Section~\ref{sec:ExPre}  a
counter-example to the original proof of Theorem~\ref{thm:CentralThm}
in~\cite[Proposition~2.9]{Bar03}. Then,
in Section~\ref{sec:StabSubgrps}, we introduce the stabilizing
subgroups which are the main tools in our proof
of Theorem~\ref{thm:CentralThm}.
In Section~\ref{sec:InvSubgrpLpres}, we study conditions on
the finite index subgroup of an invariantly $L$-presented group to be
invariantly $L$-presented itself. We conclude this paper by considering
two examples of subgroup $L$-presentations in Section~\ref{sec:Ex}
including the normal closure of a generator of the Grigorchuk group
considered in~\cite{BG02,Gr05}. We fix a mistake in the generating set
of the normal closure $\la d\ra^\Grig$ using our Reidemeister-Schreier
theorem for finitely $L$-presented groups. Therefore we show, in 
the style of~\cite{HR94}, how these computational methods can be applied
in the investigation of self-similar groups. 

\section{Preliminaries}\Label{sec:Prel}
In the following, we briefly recall the notion of a finite
$L$-presentation and the notion a finitely $L$-presented group as introduced
in~\cite{Bar03}. Moreover, we recall some basic constructions for finite
$L$-presentations.\smallskip

A \emph{finite $L$-presentation} is a group presentation of the 
form 
\begin{equation}\Label{eqn:FinLp}
  \Big\la \X\,\Big|\, \Q \cup \bigcup_{\sigma\in\Phi^*} \R^\sigma \Big\ra,
\end{equation}
where $\X$ is a finite alphabet, $\Q$ and $\R$ are finite subsets
of the free group $F$ over $\X$, and $\Phi^* \subseteq \End(F)$
denotes the free monoid of endomorphisms which is finitely generated by
$\Phi$. We also write $\la \X\mid\Q \mid\Phi\mid\R\ra$
for the finite $L$-presentation in Eq.~(\ref{eqn:FinLp}) and $G = \la
\X\mid\Q\mid\Phi\mid\R\ra$ for the finitely $L$-presented group.\smallskip

A group which admits a finite $L$-presentation is \emph{finitely
$L$-presented}. An $L$-presentation of the form $\la\X\mid\emptyset\mid\Phi\mid\R\ra$ is an 
\emph{ascending $L$-presentation} and an $L$-presentation $\la\X\mid\Q\mid\Phi\mid\R\ra$
is \emph{invariant} (and the group it presents
is \emph{invariantly $L$-presented}), if each endomorphism $\varphi \in
\Phi$ induces an endomorphism of the group $G$; that is, if the normal
subgroup $\la \Q \cup \bigcup_{\sigma\in\Phi^*} \R^\sigma \ra^F$ is
$\varphi$-invariant. Each ascending $L$-presentation is invariant
and each invariant $L$-presentation $\la \X\mid\Q\mid\Phi\mid\R\ra$
admits an ascending $L$-presentation $\la \X\mid\emptyset\mid
\Phi\mid\Q\cup\R\ra$ which defines the same group. On the other hand, we 
have the following
\begin{proposition}\Label{prop:NonInvLp}
  There are finite $L$-presentations that are not invariant.
\end{proposition}
\begin{proof}
  The group $B = \la\{a,b,t\}\mid \{a^t\,a^{-4}, b^{t^{-1}}\,b^{-2},
  [a,b^{t^i}] \mid i\in\Z\}\ra$ is a met-abelian, infinitely related
  group with trivial Schur multiplier~\cite{Bau71}. By introducing 
  a \emph{stable letter} $u$, this group admits the finite
  $L$-presentation 
  \[
    \la\{a,b,t,u\} \mid \{ub^{-1}\} \mid \{\sigma,\delta\} \mid
    \{ a^t a^{-4}, b^{t^{-1}} b^{-2}, [a,u] \} \ra, 
  \]
  where $\sigma$ is the free group homomorphism induced by the map
  $\sigma\colon a\mapsto a$, $b\mapsto b$, $t\mapsto t$, and $u \mapsto u^t$, 
  while $\delta$ is the free group homomorphism induced by the map 
  $\delta\colon a\mapsto a$, $b\mapsto b$, $t\mapsto t$, and $u \mapsto u^{t^{-1}}$. 
  This finite $L$-presentation is not invariant~\cite{Har08}.
\end{proof}
The class of finitely $L$-presented groups contains all finitely 
presented groups: 
\begin{proposition}\Label{prop:AscInvLp}
  Each finitely presented group $\la\X\mid\R\ra$ is finitely
  $L$-pre\-sen\-ted by the invariant (or ascending) finite $L$-presentation
  $\la\X\mid\emptyset\mid\emptyset\mid\R\ra$.
\end{proposition}
Therefore, (invariant or ascending) finite $L$-presentations
generalize the concept of finite presentations.  Examples of finitely
$L$-presented, but not finitely presented groups, are various self-similar
or branch groups~\cite{Bar03} including the Grigorchuk group~\cite{Gri80,Lys85,Gri99} and its twisted
twin~\cite{BS10}. However, the concept of a
finite $L$-presentation is quite general so that other examples of
infinitely presented groups are finitely $L$-presented as well. For
instance, the groups in~\cite{Bau71,KW01,OS02} are all finitely
$L$-presented.\smallskip

Various group theoretic constructions that preserve the property of being
finitely $L$-presented have been studied in~\cite{Bar03}. For completeness, 
we recall some of these constructions in the remainder of this section.
\def\0{\cite[Proposition~2.7]{Bar03}}
\begin{proposition}[\0]\Label{prop:FpExt}
  Let $G = \la \X\mid\Q\mid\Phi\mid \R\ra$ be a finitely $L$-presented
  group and let $H = \la {\mc Y}\mid \S \ra$ be finitely presented. The
  group $K$ which satisfies the short exact sequence $1 \to G \to K \to H
  \to 1$ is finitely $L$-presented.
\end{proposition}
\begin{proof}
  We recall the constructions from~\cite{Bar03} in the following:
  Let $\delta\colon H \to K$ be a section of $H$ to
  $K$ and identify $G$ with its image in $K$.
  Each relation $r \in \S$ of the finitely presented group $H$
  lifts, through the section $\delta$, to an
  element $g_r \in G$. As the group $G$ is normal in $K$, each generator
  $t \in {\mc Y}$ of the finitely presented group $H$ acts, via $\delta$, on the subgroup $G$. Thus we have $x^{\sigma(t)} = g_{x,t}
  \in G$ for each $x \in \X$ and $t \in {\mc Y}$. If $\X\cup {\mc Y}
  = \emptyset$, we may consider the 
  following finite $L$-presentation
  \begin{equation}\Label{eqn:FpExt}
    \la \X\cup {\mc Y} \mid \Q\cup\{ r\,g_r^{-1} \mid r \in \S \} \cup
    \{x^t g_{x,t}^{-1} \mid x\in \X, t\in {\mc Y}\} \mid \widehat\Phi \mid 
    \R \ra,
  \end{equation}
  where the endomorphisms $\Phi$ of $G$'s finite $L$-presentation
  $\la\X\mid\Q\mid\Phi\mid\R\ra$ are extended to endomorphisms
  $\widehat\Phi = \{\widehat\sigma \mid \sigma\in\Phi\}$ of the free
  group $F(\X\cup{\mc Y})$ by
  \[
    \widehat\sigma\colon F(\X\cup{\mc Y}) \to F(\X\cup{\mc Y}),\:
    \left\{\begin{array}{rcll}
      x &\mapsto& x^\sigma, &\textrm{for each }x\in\X\\
      y &\mapsto& y, &\textrm{for each }y\in{\mc Y}.
    \end{array}\right.
  \]
  The finite $L$-presentation in Eq.~(\ref{eqn:FpExt}) 
  is a presentation for $K$; see~\cite{Bar03}.
\end{proof}
As a finite group is finitely presented, Proposition~\ref{prop:FpExt} yields the 
immediate
\begin{corollary}\Label{cor:FinExt}
  Each finite extension of a finitely $L$-presented group is
  finitely $L$-presented.
\end{corollary}
Note that the constructions in the proof of Proposition~\ref{prop:FpExt}
above give a finite $L$-presentation for $K$ which is not ascending --
even if the group $G$ is given by an ascending $L$-presentation.  We therefore
ask the following
\begin{question}
  Is every finite extension of an invariantly (finitely) $L$-presented group 
  invariantly (finitely) $L$-presented?
\end{question}
We do not have an answer to this question in general; though we
suspect its answer is negative, see Remark~\ref{rem:LimOfProof}.
Given endomorphisms $\Phi$ of the normal subgroup $G$ in
Proposition~\ref{prop:FpExt}, one problem is to construct endomorphisms
of the finite extension $K$ which restrict to $\Phi$. This does not seem
to be possible in general.\smallskip

A finite $L$-presentation for a free product of two finitely $L$-presented
groups is given by the following improved version
of~\cite[Proposition~2.6]{Bar03}.
\begin{proposition}\Label{prop:FreeProd}
  The free product of two finitely $L$-presented groups is finitely 
  $L$-presented. If both finitely $L$-presented groups are invariantly 
  $L$-presented, then so is their free product.
\end{proposition}
\begin{proof}
  Although a proof of the first claim can be found in~\cite{Bar03}, we summarize its
  construction for our proof of the second claim. Let $G = \la \X\mid\Q\mid\Phi\mid
  \R\ra$ and $H = \la{\mc Y}\mid {\mc S}\mid \Psi\mid {\mc T}\ra$ be
  finitely $L$-presented groups. Suppose that $\X\cap{\mc Y} = \emptyset$
  holds. Then $G * H$ is finitely $L$-presented by
  $\la\,\X\cup {\mc Y} \mid \Q \cup {\mc S} \mid \tilde\Phi \cup \tilde\Psi
  \mid R\cup {\mc T}\,\ra$\linebreak (see~\cite{Bar03}),
  where the endomorphisms in $\Phi$ and in $\Psi$ are extended to
  endomorphisms $\widehat\Phi$ and $\widehat\Psi$ of the free group $F(\X\cup{\mc
  Y})$ over $\X \cup {\mc Y}$ as follows: for each $\sigma \in \Phi$,
  we let
  \[
    \widehat\sigma \colon F(\X \cup {\mc Y}) \to F(\X \cup {\mc Y} ),\:
    \left\{ \begin{array}{rcll}
       x & \mapsto & x^\sigma,&\textrm{ for each }x\in\X\\
       y & \mapsto & y,&\textrm{ for each }y\in{\mc Y}; 
    \end{array}\right.
  \]
  and, accordingly, for each $\delta\in \Psi$. As an invariant
  $L$-presentation $\la\X\mid\Q\mid\Phi\mid\R\ra$ can be considered
  as an ascending $L$-presentation $\la\X\mid\emptyset\mid\Phi\mid\Q\cup\R\ra$, we 
  can consider $\Q$ and $\S$ to be empty. Then the
  latter construction from~\cite{Bar03} shows that the free product $G*H$
  is ascendingly finitely $L$-presented and thus it is invariantly finitely $L$-presented.
\end{proof}
We further have the following improved version of~\cite[Proposition~2.9]{Bar03}:
\begin{proposition}\Label{prop:Factor}
  Let $N \unlhd G$ be a normal subgroup of a finitely $L$-presented
  group $G = \la\X\mid\Q\mid\Phi\mid\R\ra$. If $N$ is finitely
  generated as a normal subgroup, the factor group $G/N$ is finitely
  $L$-presented. If, furthermore, $G$ is invariantly $L$-presented and the normal subgroup $N$ is invariant
  under the induced endomorphisms $\Phi$, then $G/N$ is invariantly
  $L$-presented.
\end{proposition}
\begin{proof}
  Let $N = \la g_1,\ldots,g_n\ra^G$ be a finite normal generating set of
  the normal subgroup $N$. We consider the generators $g_1,\ldots,g_n$
  as elements of the free group $F$ over $\X$. Then the finite
  $L$-presentation $\la\X\mid\Q\cup\{g_1,\ldots,g_n\}\mid\Phi\mid\R\ra$
  is a finite $L$-presentation for the factor group $G/N$;
  see~\cite{Bar03}. Suppose that $G$ is given by an 
  invariant $L$-presentation $\la\X\mid\Q\mid\Phi\mid\R\ra$. Then
  $G = \la \X\mid \emptyset \mid\Phi\mid\Q\cup\R\ra$. As $N ^ \sigma 
  \subseteq N$ holds, each $\sigma\in\Phi^*$ induces an endomorphism 
  of the $L$-presented factor group $G/N$. Thus the images 
  $g_1^\sigma, \ldots, g_n^\sigma$ are consequences of the relations
  of $G/N$'s finite $L$-presentation. Therefore
  $G/N \cong \la\X\mid\Q\mid\Phi\mid\R\cup\{g_1,\ldots,g_n\}\ra
  = \la\X\mid\emptyset\mid\Phi\mid\Q\cup\R\cup\{g_1,\ldots,g_n\}\ra$.
\end{proof}
Note that, if $G$ is invariantly $L$-presented and $N$ is a
normal $\Phi$-invariant subgroup, then, in the notion of Krull \&
Noether~\cite{Kr25,Noe29}, the group $G$ is a group with operator domain
$\Phi$ and the normal subgroup $N$ is an admissible subgroup.
Proposition~\ref{prop:FreeProd} and
Proposition~\ref{prop:Factor} yield the following straightforward
\begin{corollary} \Label{cor:Amalg}
  Let $G$ and $H$ be finitely $L$-presented groups and let $F$ be a finitely
  generated group with isomorphisms $\psi\colon F \to G$ and $\phi\colon
  F \to H$. Then the amalgamated free product $G *_F H$ is finitely
  $L$-presented.
\end{corollary}
For further group theoretic constructions which preserve the property of 
being finitely $L$-presented were refer to~\cite{Bar03}.\bigskip

\section{The Reidemeister-Schreier process}\Label{sec:ReidSchr}
In the following, we briefly recall the Reidemeister-Schreier 
process for finite index subgroups as, for instance, outlined in~\cite{LS77,Sim94}. For this purpose, let $G$ be a group given by a group
presentation $\la\X \mid {\mc K}\ra$ where $\X$ is a (finite) alphabet
which defines the free group $F$ and ${\mc K} \subseteq F$ is a (possibly
infinite) set of relations. Denote the normal closure of ${\mc K}$ in $F$
by $K = \la {\mc K}\ra^F$. Then $G = F/K$.\smallskip

Let $\U\leq G$ be a finite index subgroup of $G$ given by its generators
$g_1,\ldots,g_n$. Let $T\subseteq F$ be a Schreier transversal for $\U$ in
$G$ (i.e., a transversal for $\U$ in $G$ so that every initial segment of
an element of $T$ itself belongs to $T$, see~\cite{LS77}; note that we always 
acts by multiplication from the right). We consider the
generators of $\U$ as words over the alphabet $\X$ and thus as elements
of the free group $F$. Then the subgroup $U = \la g_1,\ldots,g_n\ra$
satisfies that $\U \cong \UK/K$. In the style of~\cite{LS77}, we define
the \emph{Schreier map} $\gamma\colon T\times\X \to F$ by $\gamma(t,x) =
tx\,(\overline{tx})^{-1}$ where $\overline{tx}$ denotes the unique element
$s\in T$ from the Schreier transversal so that $\UK\,s = \UK\,tx$ holds.
The Schreier theorem (as, for instance, in~\cite[Proposition~I.3.7]{LS77}) shows that the subgroup $\UK \leq F$ is freely
generated by the \emph{Schreier generating set}
\[
  {\mc Y} = \{ \gamma(t,x) \neq 1 \mid t\in T, x\in\X\}.
\]
In particular, the Schreier theorem yields that a finite index subgroup of a finitely generated
group is itself finitely generated. We consider the set ${\mc Y}$
as an alphabet and we denote by $F({\mc Y})$ the free group over ${\mc Y}$. 
The \emph{Reidemeister rewriting} $\tau$ is a map $\tau\colon
F \to F({\mc Y})$ given by
\[
  \tau( y_1\cdots y_n )
  = \gamma(1,y_1)\cdot\gamma(\overline{y_1},y_2)\cdots \gamma(\overline{y_1\cdots y_{n-1}},y_n)
\]
where each $y_i\in\X\cup\X^-$. In general, the Reidemeister rewriting $\tau$
is not a group homomorphism; though, we have the following
\begin{lemma}\Label{lem:ReidMap}
  For $H \leq \UK$, the restriction $\tau\colon H \to F({\mc Y})$ 
  is a homomorphism.
\end{lemma}
\begin{proof}
  Let $g,h\in H$ be given. Write $g = g_1\cdots g_n$ and $h = h_1 \cdots h_m$
  with each $h_i,g_j\in \X\cup\X^-$. Then, as $\overline{g_1\cdots g_n} =
  \overline{g} = 1$ holds, we obtain that
  \begin{eqnarray*}
   \tau(gh) = \gamma(1,g_1)\cdots\gamma(\overline{g_1\cdots g_{n-1}},g_n)
   \cdot \gamma(1,h_1)\cdots \gamma(\overline{h_1\cdots h_{m-1}},h_m)
   = \tau(g)\,\tau(h)
  \end{eqnarray*}
  while we already have $\tau(1) = 1$ by definition.
\end{proof}
By Schreier's theorem, the Reidemeister rewriting $\tau\colon \UK \to F({\mc Y})$ gives an
isomorphism of free groups. A group presentation for the subgroup
$\U\cong \UK/K$ is given by the following well-known theorem; cf.~\cite[Section~II.4]{LS77}.
\begin{theorem}[Reidemeister-Schreier Theorem]\Label{thm:ReidSchr}
  If $\tau$ denotes the Reide\-meis\-ter-Schreier rewriting, $T$ denotes a Schreier
  transversal for $\U$ in $G$, and if $\la\X\mid{\mc K}\ra$ is a presentation for $G$,
  the subgroup $\U$ is presented by
  \begin{equation}\Label{eqn:SubgrpPres}
    \U \cong
    \left\la\,{\mc Y} \mid \{ \tau(trt^{-1})\mid r\in{\mc K},t\in T\}\,\right\ra.
  \end{equation}
\end{theorem}
\begin{proof}
  We recall the proof for completeness:
  Notice that $\U \cong \UK/K \cong \tau(\UK)/\tau(K)$ holds. By
  Schreier's theorem, we have $\tau(\UK) = F({\mc
  Y})$. It therefore suffices to determine a normal generating set for 
  $\tau(K)$. As ${\mc K}$ is a normal generating set
  for $K\unlhd F$, a generating set for the image $\tau(K)$ is given by
  $ \tau(K) = \la\{\tau(grg^{-1})\mid r\in{\mc K},g\in F\}\ra$.
  Since $T$ is a transversal for $\UK$ in $F$, each $g\in F$ can be
  written as $g = u\,t$ with $t\in T$ and $u\in\UK$. This yields
  $ \tau(K) = \la\{\tau(utrt^{-1}u^{-1}) \mid r\in{\mc K}, g=ut\in F\}\ra$.
  For each relation $r\in{\mc K}$, we have that $trt^{-1} \in \UK$
  and $u\in\UK$. By Lemma~\ref{lem:ReidMap}, we obtain that $\tau(
  utrt^{-1}u^{-1} ) = \tau(u)\,\tau(trt^{-1} )\,\tau(u)^{-1}$. Therefore
  $\tau(utrt^{-1}u^{-1})$ is a consequence of $\tau(trt^{-1})$ and hence, it 
  can be omitted. 
  Thus a normal generating set for $\tau(K)$ is given by
  $\tau(K) = \left\la\{\tau(trt^{-1})\mid r\in {\mc K},t\in T\}\right\ra ^ {F({\mc Y})}$.
\end{proof}
In particular, if $\U$ is a finite index subgroup of a finitely
presented group $G$, there exist a finite set of relations ${\mc K}$
and a finite Schreier transversal $T$ so that the subgroup $\U$ is
finitely presented by Theorem~\ref{thm:ReidSchr}. This latter result
for finitely presented groups is well-known and it is often simply referred to
the Reidemeister-Schreier theorem for finitely presented groups. In
this paper, we prove a variant of the Reidemeister-Schreier theorem
for finitely $L$-presented groups.

\section{A typical example of a subgroup {\boldmath$L$}-presentation}\Label{sec:ExPre}
Before proving Theorem~\ref{thm:CentralThm}, we first consider an 
example of a finite $L$-pre\-sen\-tation for a finite index subgroup of a
finitely $L$-presented group. For this purpose we consider a subgroup
of the Basilica group~\cite{GZ02b}. The Basilica group satisfies the following
\def\0{\cite{BV05}}
\begin{proposition}[Bartholdi \& Vir{\'a}g, \0]\Label{prop:Basilica}
  The Basilica group $G$ is invariantly finitely $L$-presented by
  $G \cong \la \{a,b\} \mid \emptyset \mid \{\sigma\} \mid \{[a,a^b]\}\ra$
  where $\sigma$ is the free group homomorphism induced by the map
  $a \mapsto b^2$ and $b\mapsto a$.
\end{proposition}
Since the Basilica group $G$ is invariantly $L$-presented,
the substitution $\sigma$ induces an endomorphism of $G$. The 
group $G$ will often provide an exclusive (counter-) example throughout this
paper.\smallskip

Consider the subgroup $\U = \la a, bab^{-1}, b^3\ra$ of the
Basilica group. Then coset-enumeration for finitely $L$-presented
groups~\cite{Har10b} shows that $\U$ is a normal subgroup of $G$ with
index $3$.  A Schreier generating set for the subgroup $\U$ is given by
$\{a,bab^{-1},b^2ab^{-2},b^3\}$. Write $x_1 = a$, $x_2 = bab^{-1}$, $x_3
= b^2 a b^{-2}$, and $x_4 = b^3$. Denote the free group over $\{a,b\}$
by $F$ and let ${\mc F}$ denote the free group over $\{x_1,x_2,x_3,x_4\}$.
For each $n\in\N_0$, we define $a_n = (2^n+2)/3$ and $b_n = (2^n+1)/3$.
Then the $\sigma$-images of the iterated relation $r = [a,a^b]$ can be
rewritten with the Reidemeister rewriting $\tau\colon F \to {\mc
F}$. Their images have the form
\[
  \tau( r^{\sigma^{2n}} ) = \left\{ \begin{array}{cl}
    \left[x_1^{2^n},\, x_4^{-a_n}\,x_3^{2^n}\,x_4^{a_n}\right], &\textrm{if $n$ is even},\\[1ex]
    \left[x_1^{2^n},\, x_4^{-b_n}\,x_2^{2^n}\,x_4^{b_n}\right], &\textrm{if $n$ is odd},
  \end{array}\right.
\]
and
\[
  \tau( r^{\sigma^{2n+1}} ) = \left\{ \begin{array}{cl}
    x_4^{-b_{n+1}} x_2^{-2^n} x_4^{-b_{n+1}} x_3^{2^n}\,x_4^{b_{n+1}}\,x_2^{-2^n} x_4^{b_{n+1}}\,x_1^{2^n},&\textrm{if $n$ is even},\\[0.7ex]
    x_4^{-a_{n+1}} x_3^{-2^n} x_4^{-a_{n+1}+1} x_2^{2^n}\,x_4^{a_{n+1}-1}\,x_3^{-2^n} x_4^{a_{n+1}}\,x_1^{2^n}, &\textrm{if $n$ is odd}.
  \end{array}\right.
\]
Note that $\tau(r^{\sigma^{2n}}) \in [{\mc F},{\mc F}]$ though 
$\tau( r^{\sigma^{2n+1}} ) \not \in [{\mc F},{\mc F}]$. Therefore,
the images $\tau(r^{\sigma^i})$ split into two classes which are recursive
images of the endomorphism 
\[
  \widehat\sigma \colon {\mc F} \to {\mc F}, \left\{\begin{array}{rcl}
    x_1 &\mapsto& x_1^2,\\
    x_2 &\mapsto& x_3^2,\\
    x_3 &\mapsto& x_4\,x_2^2\,x_4^{-1},\\
    x_4 &\mapsto& x_4^2;
  \end{array}\right. 
\]
in the sense that $\widehat\sigma$ satisfies
\[ 
   \tau( r^{\sigma^{2n}} ) = [x_1,x_4^{-1} x_3\,x_4 ]^{\widehat\sigma^{n}}
   \textrm{ and }
   \tau( r^{\sigma^{2n+1}} ) =
   ( x_4^{-1} x_2^{-1} x_4^{-1} x_3\,x_4\,x_2^{-1} x_4\,x_1 )^{\widehat\sigma^{n}},
\]
for each $n\in\N_0$. In Section~\ref{sec:Ex}, we will show that a finite
$L$-presentation for the subgroup $\U$ is given by
\[
  \U \cong \left\la \{x_1,\ldots,x_4\}\,\middle|\, \emptyset \,\middle|\,
  \{ \widehat\sigma,\delta \} \,\middle|\,
  \{[x_1,x_4^{-1} x_3\,x_4 ],
   x_4^{-1} x_2^{-1} x_4^{-1} x_3\,x_4\,x_2^{-1} x_4\,x_1\}
   \right\ra
\]
where the endomorphism $\delta$ is induced by the mapping
\[
  \delta\colon {\mc F} \to {\mc F}, \left\{\begin{array}{rcl}
    x_1 &\mapsto& x_2,\\
    x_2 &\mapsto& x_3,\\
    x_3 &\mapsto& x_4\,x_1\,x_4^{-1},\\
    x_4 &\mapsto& x_4.
  \end{array}\right. 
\]
These subgroup $L$-presentations are typical for finite index subgroups
of a finitely $L$-presented group. Besides, the subgroup $\U$ and
its subgroup $L$-pre\-sen\-tation provide a counter-example to the original proof
of Theorem~\ref{thm:CentralThm} in~\cite{Bar03} as there is no
endomorphism $\varepsilon$ of the free group ${\mc F}$ such that $\tau(
r^{\sigma^{n+1}} ) = \varepsilon( \tau( r^{\sigma^{n}} ) )$ for each
$n\in\N_0$. A reason for the failure of the proof in~\cite{Bar03} is
that the subgroup $\U$ is not $\sigma$-invariant but $\sigma^2$-invariant.
Therefore, the method suggested in the proof of~\cite[Proposition~2.9]{Bar03}
will fail to compute a finite $L$-presentation for $\U$.

\section{Stabilizing subgroups}\Label{sec:StabSubgrps}
In this section, we introduce the stabilizing subgroups which will be 
central to what follows.\smallskip

Let $G = \la\X\mid\Q\mid\Phi\mid\R\ra$ be a finitely $L$-presented group
and let $\U$ be a finite index subgroup of $G$ which is generated by
$g_1,\ldots,g_n$, say. Denote the free group over $\X$ by $F$ and let
$K = \la \Q\cup\bigcup_{\sigma\in\Phi^*} \R^\sigma\ra^F$. We consider
the generators $g_1,\ldots,g_n$ of the subgroup $\U$ as words over
the alphabet $\X$.  Thus the subgroup $U = \la g_1,\ldots,g_n \ra$
of the free group $F$ satisfies $\U \cong  \UK/K$. The group $F$ acts
on the right-cosets $\UK \bs F$ by multiplication from the right. Let
$\varphi\colon F \to \Sym(\UK\bs F)$ be a permutation representation for
the group's action on $\UK\bs F$. Note that this permutation representation can
be computed with the coset-enumeration methods in~\cite{Har10b}. 
We obtain the following
\begin{lemma}
  The kernel $\ker(\varphi)$ is the normal core, $\Core_F(\UK)$, of 
  $\UK$ in $F$. 
\end{lemma}
\begin{proof}
  As each $g \in \ker(\varphi)$ stabilizes the right-coset
  $\UK\,1$, the kernel $\ker(\varphi) \unlhd F$ 
  is contained in the subgroup $\UK = \varphi^{-1}(\Stab_{\Sym(\UK\bs F)}(\UK\,1))$. Hence $\ker(\varphi) \leq \Core_F(\UK)$.
  Recall that $\Core_F(\UK) = \bigcap_{x\in F} x\,\UK\,x^{-1}$.
  We show that each $g \in \Core_F(\UK)$ acts trivially on
  the right-cosets $\UK\bs F$. Let $t \in T$
  be given. Then, as $F$ acts transitively on the cosets $\UK\bs F$, there exists
  $h\in F$ so that $t\,h = v\in \UK$. Then $h^{-1} = v^{-1} t$. As $g
  \in \bigcap_{x\in F} x\,\UK\,x^{-1}$, there exists $u \in \UK$
  with $g = h\,u\,h^{-1}$. Hence $\UK\,t \cdot g = \UK\,t\,huh^{-1} =
  \UK\,u\,h^{-1} = \UK\,h^{-1} = \UK\,v^{-1}\,t = \UK\,t$ and so
  $g$ acts trivially on the right-cosets $\UK\bs F$. Thus we have 
  $g \in \ker(\varphi)$ and $\Core_F(\UK) \leq \ker(\varphi)$.
\end{proof}
In the following we define the stabilizing subgroups. 
These subgroups will be central to our 
proof of Theorem~\ref{thm:CentralThm} in Section~\ref{sec:ReidSchrThm}.
\begin{definition}\Label{def:StabSubs}
  Let $G = \la\X\mid\Q\mid\Phi\mid\R\ra$ be a finitely $L$-presented group and let
  $\U \leq G$ be a finite index subgroup which admits the permutation representation
  $\varphi \colon F \to \Sym(\UK\bs F)$.  The \emph{stabilizing subgroup}
  of $\U$ is
  \begin{equation}\Label{eqn:StabSub}
    \ti\El = \bigcap_{\sigma\in\Phi^*} (\sigma\varphi)^{-1}( \Stab_{\Sym(\UK\bs F)}(\UK\,1)).
  \end{equation}
  The \emph{stabilizing core} of $\U$ is
  \begin{equation}
    \El = \bigcap_{\sigma\in\Phi^*} \ker(\sigma\varphi).
  \end{equation}
\end{definition}

For each $\sigma\in \Phi^*$, we denote by $\|\sigma\|$ the
usual word-length in the generating set $\Phi$. The free
monoid $\Phi^*$ has the structure of a $|\Phi|$-regular tree with
its root being the identity map $\id\colon F\to F$. We can further endow 
the monoid $\Phi^*$ with a length-plus-(from the right)-lexicographic
ordering $\prec$ by choosing an arbitrary ordering on the (finite) generating set
$\Phi$. We then define $\sigma \prec \delta$ if $\|\sigma\| < \|\delta\|$
or, otherwise, if $\sigma = \sigma_1\cdots \sigma_n$ and $\delta =
\delta_1 \cdots \delta_n$, with each $\sigma_i,\delta_j \in \Phi$, and
there exists a positive integer $1\leq k\leq n$  such that $\sigma_i =
\delta_i$ for each $k< i\leq n$, and $\sigma_k \prec \delta_k$. Since
$\Phi$ is finite, the constructed ordering $\prec$ is a well-ordering on
the monoid $\Phi^*$; see~\cite{Sim94}. Thus, there is no infinite
descending sequences $\sigma_1 \succ \sigma_2 \succ \ldots$ in $\Phi^*$. 

We consider Algorithm~\ref{alg:IteratingEndomorphisms} below. 
\begin{algorithm}[htp]
  \caption{Computing a finite set of endomorphisms $V\subseteq \Phi^*$; see also~\cite{Har10b}}
  \label{alg:IteratingEndomorphisms}
  \begin{center}
  {\begin{minipage}{10cm}
    \begin{tabbing}
      {\scshape IteratingEndomorphisms}($\X$, $\Q$, $\Phi$, $\R$, $\U$, $\varphi$)\\
      \quad\= Initialize $S:=\Phi$ and $V:=\{\id\colon F\to F\}$.\\
      \> Choose an ordering on $\Phi = \{\phi_1,\ldots,\phi_n\}$ with $\phi_i \prec \phi_{i+1}$.\\
      \>{\bf while} $S \neq \emptyset$ {\bf do} \\
      \>\quad\= Remove the first entry $\delta$ from $S$.\\
      \>\>{\bf if not} $\left( \exists\,\sigma\in V\colon\:{\delta\varphi = \sigma\varphi}\right)$ {\bf then}\\
      \>\>\quad\= Append $\phi_1\delta,\ldots,\phi_n\delta$ to $S$.\\
      \>\>\> Add $\delta$ to $V$. \\
      \>{\bf return}( $V$ ) 
    \end{tabbing}
  \end{minipage}}
  \end{center}
\end{algorithm}
If $\varphi\colon F \to \Sym(\UK\bs F)$ denotes a permutation
representation as in Definition~\ref{def:StabSubs}, the algorithm {\scshape
IteratingEndomorphisms} returns a finite image of a section of the
map $\Phi^* \to \Hom(F,\Sym(\UK\bs F))$ defined by $\sigma \mapsto
\sigma\varphi$. More precisely, we have the following 
\begin{lemma}\Label{lem:Term}
  The algorithm {\scshape IteratingEndomorphisms} terminates and it returns
  a finite set of endomorphisms $V \subseteq \Phi^*$ satisfying the
  following property: For each $\sigma_1\in\Phi^*$ there exists a unique
  $\sigma_n \in V$ so that $\sigma_1\varphi = \sigma_n\varphi$. The 
  element $\sigma_n$ is minimal with respect to the total ordering $\prec$ constructed above. 
\end{lemma}
\begin{proof}
  Let $\X$ be a basis of the free group $F$. Then a homomorphism
  $\psi\colon F \to \Sym(\UK\bs F)$ is uniquely defined by the image
  of this basis. 
  Since $\UK\bs F$ is finite, the symmetric
  group  $\Sym(\UK\bs F)$ is finite. Moreover, as $F$ is finitely
  generated, the set of homomorphisms $\Hom(F,\Sym(\UK\bs F))$ is
  finite. Therefore the algorithm {\scshape IteratingEndomorphisms}
  can add only finitely many elements to $V$.  Thus the stack $S$ will
  eventually be reduced and the algorithm terminates.\smallskip

  The ordering $\prec$ on $\Phi$ can be extended to a total and well-ordering
  on the free monoid $\Phi^*$ as described above. The elements in the stack 
  $S$ are always ordered with respect to the total and well-ordering $\prec$. They further
  always succeed those elements in $V$. In particular, the elements in $V$ are 
  minimal.
  Let $\sigma_1 \in \Phi^*$ be given. Then
  there exists $w\in\Phi^*$ maximal subject to the existence of
  $\delta \in V$ so that $\sigma_1 = w\delta$. If $\|w\| = 0$ holds,
  then $\sigma_1 \in V$ and the claim is proved.  Otherwise, there
  exists $\psi \in \Phi$ so that $\sigma_1 = v \psi \delta$ for some
  $v\in\Phi^*$ and $\psi\delta \not\in V$. The algorithm yields the
  existence of $\varepsilon\in V$ so that $\varepsilon \prec \psi\delta$
  and $\psi\delta\varphi = \varepsilon\varphi$.  We also have that
  $\sigma_2 = v\varepsilon \prec v\psi\delta = \sigma_1$.  This rewriting
  process yields a descending sequence $\sigma_1 \succ \sigma_2 \succ
  \ldots$ of endomorphisms. As $\prec$ is a well-ordering there exists
  $\sigma_n \in V$ so that $\sigma_1\succ \sigma_2 \succ\ldots \succ
  \sigma_n$ and $\sigma_1\varphi = \sigma_n\varphi$. Clearly, the element
  $\sigma_n$ is unique.
\end{proof}
If $\varphi\colon F \to \Sym(\UK\bs F)$ is a permutation
representation for an infinite index subgroup $\UK\leq F$, 
we cannot ensure finiteness of the set $V$ above.\smallskip

For finite $L$-presentations $\la\X\mid\Q\mid\Phi\mid\R\ra$ with
$\Phi = \{\sigma\}$, Algorithm~\ref{alg:IteratingEndomorphisms}
and Lemma~\ref{lem:Term} yield the following
\begin{corollary}
  If $\Phi = \{\sigma\}$, there exist integers $0\leq i<j$
  with $\sigma^j\varphi = \sigma^i \varphi$.
\end{corollary}
The set $V\subseteq \Phi^*$ returned by Algorithm~\ref{alg:IteratingEndomorphisms}
satisfies the following
\begin{lemma}
  The set $V$ can be considered as a subtree of $\Phi^*$. The image of the finite set $V$ and the image of
  the monoid $\Phi^*$ in $\Hom(F,\Sym(\UK\bs F))$ coincide.
\end{lemma}
\begin{proof}
  The identity mapping $\id\colon F\to F$ is contained in the
  set $V$ and it represents the root of $V$.  Let $\sigma\in V$ be given. Then either $\sigma\in \Phi$ or
  there exists $\psi\in\Phi$ and $\delta \in \Phi^*$ so that $\sigma =
  \psi\delta$. In the first case, the identity mapping $\id\colon F \to
  F$ is a unique parent of $\sigma \in \Phi$. Suppose that $\sigma =
  \psi \delta$ holds. We need to show that $\delta \in V$. The
  algorithm {\scshape IteratingEndomorphisms} only adds elements from the
  stack $S$ to $V$. Thus at some stage of the algorithm we had $\sigma =
  \psi \delta \in S$; however, this element is added to the stack $S$ as a child
  of the element $\delta$. The uniqueness follows from the freeness of
  $\Phi^*$. The second argument follows immediately from
  Algorithm~\ref{alg:IteratingEndomorphisms} and Lemma~\ref{lem:Term}.
\end{proof}
We define a binary relation $\sim$ on the free monoid $\Phi^*$ by defining $\sigma
\sim \delta$ if and only if the unique element $\sigma_n\in\Phi^*$ in
Lemma~\ref{lem:Term} coincides for both $\sigma$ and $\delta$. Thus
$\sigma \sim \delta$ holds if and only if $\sigma\varphi = \delta\varphi$.
This definition yields the immediate
\begin{lemma}
  The relation $\sigma \sim \delta$ is an equivalence relation. Each 
  equivalence class is represented by a unique element in $V$ which 
  is minimal with respect to the total and well-ordering $\prec$.
\end{lemma}
Recall that $\varphi\colon F \to \Sym(\UK\bs F)$ is a permutation
representation for the group's action on the right-cosets $\UK\bs
F$. Therefore, if $T$ denotes a transversal for $\UK$ in $F$, then
$\sigma \sim \delta$ implies that $\UK\,t \cdot g^{\sigma} =
\UK\,t \cdot g^\delta$, for each $t \in T$ and $g \in F$. We therefore
obtain the following
\begin{lemma}\Label{lem:InvCond}
  If $\sigma \in \Phi^*$ satisfies $\sigma\varphi = \varphi$,
  then the subgroup $\UK$ is $\sigma$-invariant. There are 
  $\sigma$-invariant subgroups $\UK$ that do not satisfy
  $\sigma\varphi = \varphi$. 
\end{lemma}
\begin{proof}
  As $\sigma\varphi = \varphi$ holds, we have $\UK\,t\cdot g^\sigma
  = \UK\,t\,g$ for each $t\in T$ and $g \in F$.  Let $g \in \UK$ be
  given. Then $\UK\,1\cdot g^\sigma = \UK\,1\cdot g = \UK\,1$ and so
  $g^\sigma \in \UK$. The index-$2$ subgroup $\U = \la a,b^2,bab^{-1} \ra$ of the
  Basilica group satisfies $(\UK)^\sigma \subseteq \UK$ and
  $\sigma\varphi \neq \varphi$. This (and similar results in the 
  remainder of this paper) can be easily verified with a computer-algebra-system 
  such as \Gap. 
\end{proof}
The latter observation motivates the following 
\begin{definition}\Label{def:LeafInv}
  Let $G = \la\X\mid\Q\mid\Phi\mid\R\ra$ be a finitely $L$-presented
  group and let $\U \leq G$ be a finite index subgroup with permutation
  representation $\varphi$. Then the \emph{$\varphi$-leafs} $\Psi \subseteq \Phi^*\setminus V$ of $V\subseteq
  \Phi^*$ are defined by
  \begin{equation}\Label{eqn:AscLpInv} 
    \Psi = \{ \psi\delta \mid \psi \in \Phi,\:\delta\in V,\:\psi\delta \not\in V,\:
    \psi\delta\varphi = \varphi \}.
  \end{equation}
  The subgroup $\U$ is \emph{leaf-invariant} if $\Psi =
  \{\psi\delta\mid \psi\in\Phi, \delta\in V, \psi\delta \not\in V\}$ holds.
\end{definition}
As $\Phi\subseteq \End(F)$ is finite and the equivalence $\sim$ yields finitely
many equivalence classes, the set of $\varphi$-leafs $\Psi$ of $V\subseteq \Phi^*$
is finite. We obtain the following
\begin{lemma}\Label{lem:AscLpInv}
  If $\U$ is a leaf-invariant subgroup of $G$, then each $\varphi$-leaf $\psi\delta
  \in \Psi$ induces an endomorphism
  of $\UK$. Moreover, each $\sigma_1 \in\Phi^*$ can be written as 
  $\sigma_1 = v\,\sigma$ with $v \in V$ and $\sigma \in \Psi^*\subseteq \End(\UK)$.
\end{lemma}
\begin{proof}
  We again follow the ideas
  of Algorithm~\ref{alg:IteratingEndomorphisms}. By
  Lemma~\ref{lem:InvCond}, the condition $\psi\sigma\varphi = \varphi$
  implies $\psi\sigma$-invariance of $\UK$ and hence $\Psi^* \subseteq
  \End(\UK)$.  Write $W = \{\psi\delta\mid \psi\in\Phi, \delta\in V, \psi\delta \not\in V\}$ and let $\sigma_1 \in \Phi^*$ be given. There exists $w
  \in \Phi^*$ maximal subject to the existence of $\delta \in V$ so that
  $\sigma_1 = w\delta$. If $\|w\| = 0$, then $\sigma_1 = \delta\cdot
  \id$ with $\delta \in V$ and $\id\in \Psi^*$.  Otherwise, there
  exists $\psi \in \Phi$ and $\sigma_2 \in \Phi^*$ so that $\sigma_1 =
  \sigma_2 \psi \delta$ and $\psi\delta \not\in V$. Note that $\psi\delta
  \in W$. Since $\U$ is leaf-invariant, we have $W = \Psi$ and hence $\psi\delta \in \Psi$.
  Therefore $\psi\delta$ induces an endomorphism of $\UK$ and we can continue with the prefix
  $\sigma_2$ of $\sigma_1$. Clearly $\sigma_2 \prec
  \sigma_1$. Rewriting the prefix $\sigma_2$ yields
  a descending sequence $\sigma_1 \succ \sigma_2 \ldots$ in $\Phi^*$. As $\prec$
  is a well-ordering, we eventually have $\sigma_1 \succ \sigma_2
  \succ \ldots\succ \sigma_n$ with $\sigma_n \in V$.
\end{proof}
If the finite $L$-presentation $\la\X\mid\Q\mid\Phi\mid\R\ra$ satisfies
$\Phi = \{\sigma\}$ and if there exists a minimal positive integer $0<j$ so that $\sigma^j\varphi = \varphi$ holds, 
then the set $W = \{\psi\delta \mid \psi\in\Phi, \delta\in V, \psi\delta
\not\in V\}$ in the proof of Lemma~\ref{lem:AscLpInv} above becomes
$W = \{\sigma^j\}$. Note the following
\begin{remark}
  The condition $\sigma^j \varphi = \sigma^0 \varphi$ is essential for the
  $\sigma^{j-0}$-invariance of the subgroup. For instance, the subgroup $\U = \la a,
  bab^{-1}, b^{-1}a^2b, b^4, b^{2}ab^{-2} \ra$ of the Basilica
  group satisfies $\sigma^4 \varphi = \sigma^3 \varphi$ but it 
  is not $\sigma$-invariant.
\end{remark}
The stabilizing subgroup $\ti\El$ introduced in Definition~\ref{def:StabSubs} satisfies the following
\begin{proposition}
  Let $V\subseteq \Phi^*$ be the finite set returned by 
  Algorithm~\ref{alg:IteratingEndomorphisms}.
  The stabilizing subgroup $\ti\El$ satisfies that
  \[
    \ti\El = \bigcap_{\sigma\in V} (\sigma\varphi)^{-1}(\Stab_{\Sym(\UK\bs F)}(\UK\,1)).
  \]
  The stabilizing subgroup $\ti\El$ is $\Phi$-invariant (i.e., we have $\ti\El^\psi \subseteq \ti\El$ for
  each $\psi\in\Phi$). It is
  contained in the subgroup $\UK$ and it has finite index in $F$. The
  stabilizing subgroup $\ti\El$ is the largest $\Phi^*$-invariant subgroup
  of $\UK$. It is not necessarily normal in $F$.
\end{proposition}
\begin{proof}
  Write ${\mc K} = \bigcap_{\sigma\in V}
  (\sigma\varphi)^{-1}(\Stab_{\Sym(\UK\bs F)}(\UK\,1))$.  Clearly
  $\ti\El \subseteq {\mc K}$ holds. Let $g\in {\mc K}$ and $\sigma
  \in \Phi^*$ be given. By Lemma~\ref{lem:Term}, there exists a unique
  $\delta\in V$ so that $\sigma\varphi = \delta\varphi$. This yields
  that $\UK\,1\cdot g^\sigma = \UK\,1 \cdot g^\delta = \UK\,1$. Thus $g
  \in (\sigma\varphi)^{-1}(\Stab_{\Sym(\UK\bs F)}(\UK\,1))$ and so ${\mc
  K} \subseteq \ti\El$.\smallskip

  Let $\psi\in\Phi$ and $\sigma\in V$ be given. Then either $\psi\sigma
  \in V$ or there exists $\delta\in V$ so that $\psi\sigma\varphi =
  \delta\varphi$.  In the first case, the image $g^{\psi\sigma}$ of the element
  $g \in\ti\El \subseteq (\psi\sigma\varphi)^{-1}(\Stab_{\Sym(\UK\bs
  F)}(\UK\,1))$ stabilizes the right-coset $\UK\,1$ and thus
  $g^{\psi\sigma} \in \UK$.  In the second case, there exists
  $\delta\in V$ so that $\psi\sigma\varphi = \delta\varphi$ and
  $g^{\psi\sigma\varphi} = g^{\delta\varphi}$. As $g \in\ti\El
  \subseteq(\delta\varphi)^{-1}(\Stab_{\Sym(\UK\bs F)}(\UK\,1))$,
  the image $g^{\delta}$ stabilizes the right-coset $\UK\,1$.
  Hence $\UK\,1\cdot g^{\psi\sigma} = \UK\,1\cdot g^{\delta} = \UK\,1$
  and thus $g^{\psi\sigma} \in \UK$. Therefore, in both cases considered
  above, we have that $g^{\psi\sigma} \in \UK$ and so $g^\psi \in
  (\sigma\varphi)^{-1}(\Stab_{\Sym(\UK\bs F)}(\UK\,1))$.  As $\sigma$
  was arbitrarily chosen, we have $g^\psi \in \ti\El$ which proves
  the $\psi$-invariance of the stabilizing subgroup $\ti\El$.\smallskip

  Let $g\in\ti\El$ be given. As $\id \in V$ holds, $g\in
  \varphi^{-1}(\Stab_{\Sym(\UK\bs F)}(\UK\,1))$. Hence the element
  $g\in\ti\El$ stabilizes the right-coset $\UK\,1$. Thus $\UK\,1\cdot
  g = \UK\,1$ and $g \in \UK$. Since $\ti\El$ is the intersection of
  finitely many finite index subgroups of $F$, the stabilizing subgroup
  has finite index in $F$.\smallskip

  Let $N$ be a $\Phi^*$-invariant subgroup satisfying $\ti\El
  \leq N \leq \UK \leq F$.  Then, for each $\sigma\in\Phi^*$, we have
  $N^\sigma \subseteq N$. Let $g \in N$ be given. Then $g^\sigma \in N \leq \UK$,
  as $N$ is $\sigma$-invariant.  Thus $\UK\,1\cdot g^\sigma = \UK\,1$
  and $g \in (\sigma\varphi)^{-1}(\Stab_{\Sym(\UK\bs F)}(\UK\,1))$.
  Since $\sigma\in\Phi^*$ was arbitrary, we have $g \in \ti\El$ and
  thus $\ti\El = N$.\smallskip

  The stabilizing subgroup $\ti\El = \la a, bab^{-1}, b^{-1}a^2b,
  b^2ab^{-2}, b^3a^{-1}b, b^{-1}ab^3 \ra$ of the subgroup
  $ \U = \la a, bab^{-1}, b^{-1}a^{-2}b, b^2ab^{-2}, b^3a^{-1}b, b^{-1}ab^3 \ra$
  of the Basilica group is not normal in $F$.
\end{proof}
The stabilizing subgroup $\ti\El$ always satisfies that $\ti\El\subseteq \UK$. 
Conditions for equality are given by the following
\begin{lemma}\Label{lem:StabSubInv}
  The stabilizing subgroup satisfies $\ti\El = \UK$ if and only if
  $(\UK)^\psi \subseteq \UK$ for all $\psi \in V$.  Moreover, we have
  $(\UK)^\psi \subseteq \UK$ for all $\psi \in V$ if and only if $(\UK)^\delta
  \subseteq \UK$ for all $\delta \in \Phi^*$.
\end{lemma}
\begin{proof}
  We already proved that $\ti\El \subseteq \UK$ holds. Let $g
  \in \UK$  and $\psi \in V$ be given. If $(\UK)^\psi \subseteq \UK$
  holds, then $\UK\,1\cdot g^\psi = \UK\,1$ and thus $g^\psi \in
  \varphi^{-1}(\Stab_{\Sym(\UK\bs F)}(\UK\,1))$.  As $\psi$ was arbitrary,
  we have $g \in \ti\El$ and therefore $\UK\subseteq\ti\El$.  
  On the other hand, suppose that $\UK \subseteq \ti\El$ holds. 
  Then, as $\ti\El$ is $\psi$-invariant, for each $\psi \in \Phi^*$, 
  we have that $(\UK)^\psi\subseteq \ti\El^\psi \subseteq \ti\El \subseteq \UK$ which 
  proves the $V$-invariance of $\UK$.\smallskip

  Clearly, as $V \subset \Phi^*$ holds, the $\Phi^*$-invariance of $\UK$
  yields the $V$-invariance of $\UK$.  On the other hand, suppose that
  $\UK$ is $V$-invariant.  Let $\delta \in \Phi^*$ be given. Then there
  exists $\sigma \in V$ so that $\sigma \varphi = \delta \varphi$.  If $g
  \in \UK$, then the image $g^\sigma$ stabilizes the right-coset $\UK\,1$
  as $\UK$ is $V$-invariant. We therefore obtain $\UK\,1\cdot g^\delta =
  \UK\,1\cdot g^\sigma = \UK\,1$ and hence $g^\delta \in \UK$ which 
  proves the $\Phi^*$-invariance of $\UK$.
\end{proof}
In the style of~\cite{Har10b}, we define a binary relation
$\leadsto_\varphi$ on the free monoid $\Phi^*$ as follows: For
$\sigma,\delta\in\Phi^*$ we define $\sigma \leadsto_\varphi \delta$ if
and only if there exists a homomorphism $\pi\colon\im(\delta\varphi)
\to \im(\sigma\varphi)$ so that $\sigma\varphi = \delta\varphi\pi$
holds. It is known~\cite{Har10b} that it is decidable whether or not 
$\sigma\leadsto_\varphi \delta$ holds. This yields that
\begin{lemma}\Label{lem:LeadstoSet}
  Let $V\subseteq \Phi^*$ be the finite set returned by
  Algorithm~\ref{alg:IteratingEndomorphisms}.  Then there exists a
  subset $\ti V \subseteq V$ with the following property: For each
  $\sigma \in \Phi^*$ there exists a unique element $\delta \in W$
  so that $\sigma\leadsto_\varphi \delta$ and $\delta$ is minimal 
  with respect to the ordering $\prec$ in Lemma~\ref{lem:Term}.
\end{lemma}
\begin{proof}
  This is straightforward as the set $V$ returned by
  Algorithm~\ref{alg:IteratingEndomorphisms} is an upper bound on $\ti V$
  because $\sigma\sim\delta$ implies both $\sigma \leadsto_\varphi \delta$
  or $\delta \leadsto_\varphi \sigma$.
\end{proof}
Again, the set $\ti V$ in Lemma~\ref{lem:LeadstoSet} can be considered 
a subtree of $\Phi^*$ or even as a subtree of $V$.
The binary relation $\leadsto_\varphi$ is reflexive and transitive but not
necessarily symmetric. The equivalence relation $\sim$ and the relation
$\leadsto_\varphi$ are related by the following
\begin{lemma}\Label{lem:SimLeadsto}
  Let $G = \la\X\mid\Q\mid\Phi\mid\R\ra$ be a  finitely $L$-presented group
  and let $\varphi\colon F \to \Sym(\UK\bs F)$ be a permutation representation.
  For $\sigma,\delta\in\Phi^*$, we have 
  \begin{enumerate}\addtolength{\itemsep}{-1ex}
  \item We have $\sigma \leadsto_\varphi \delta$ and $\delta \leadsto_\varphi
        \sigma$ if and only if the homomorphism $\pi\colon \im(\delta\varphi)
        \to \im(\sigma\varphi)$ with $\sigma\varphi = \delta\varphi\pi$
        is bijective.
  \item If $\sigma \sim \delta$, then $\sigma \leadsto_\varphi \delta$ and
        $\delta \leadsto_\varphi \sigma$. The converse is not necessarily true.
  \item If $k>0$ is minimal so that $\sigma^k \sim \id$, there exists a minimal positive integer $\ell$
        so that $\ell \mid k$ and $\sigma^\ell \leadsto_\varphi \id$. 
        If $\Phi = \{\sigma\}$, then the set $\ti V$ from Lemma~\ref{lem:LeadstoSet} becomes 
        $\ti V = \{\id,\sigma,\ldots,\sigma^{\ell-1}\}$.
  \item If $\ell$ is a minimal positive integer so that both $\sigma^\ell \leadsto_\varphi \id$ and 
        $\id \leadsto_\varphi \sigma^\ell$ hold, there exists $k \geq \ell$ so that 
        $\sigma^k \sim \id$. If $\Phi = \{\sigma\}$, then the set $V$ returned by 
        Algorithm~\ref{alg:IteratingEndomorphisms} becomes $V = \{\id,\sigma,\ldots,\sigma^{k-1}\}$ while 
        $\ti V = \{\id,\sigma,\ldots,\sigma^{\ell-1}\}$.
  \item The subgroup $\U = \la a,b^2,bab^{-1} \ra$ of the Basilica group satisfies 
        $\sigma \leadsto_\varphi \id$ but there is positive integer
        $\ell>0$ so that $\sigma^\ell \sim \id$ holds.
  \end{enumerate}
\end{lemma}
\begin{proof}
  If the homomorphism $\pi\colon \im(\delta\varphi) \to
  \im(\sigma\varphi)$ with $\sigma\varphi = \delta\varphi\pi$ is bijective, then we obtain $\sigma\varphi\pi^{-1}
  = \delta\varphi$ and thus $\delta \leadsto_\varphi \sigma$. On
  the other hand, suppose that both $\sigma \leadsto_\varphi \delta$ and
  $\delta \leadsto_\varphi \sigma$ hold. Then there are homomorphisms
  $\pi\colon \im(\sigma\varphi) \to \im(\delta\varphi)$ and $\tau\colon
  \im(\delta\varphi)\to \im(\sigma\varphi)$ so that $\delta\varphi
  = \sigma\varphi\pi$ and $\sigma\varphi = \delta\varphi\tau$. This
  yields $\delta\varphi = \sigma\varphi\pi = \delta\varphi\tau\pi$ and
  $\sigma\varphi = \delta\varphi\tau = \sigma\varphi\pi\tau$. Hence $\pi$
  and $\tau$ are isomorphisms.\smallskip

  Since $\sigma \sim \delta$ implies $\sigma\varphi = \delta\varphi$,
  we immediately obtain both $\sigma \leadsto_\varphi\delta$ and $\delta
  \leadsto_\varphi \sigma$. The subgroup $\U = \la a, bab^{-1}, b^3
  \ra$ of the Basilica group admits the permutation representation
  $\varphi\colon a \mapsto ()$, $b \mapsto (1,2,3)$. We have $\sigma^2
  \varphi\colon a\mapsto ()$, $b \mapsto (1,3,2)$ and therefore $\sigma^2
  \leadsto_\varphi \id$ and $\id \leadsto_\varphi \sigma^2$. Though
  $\sigma^2 \varphi \neq \varphi$.\smallskip

  Suppose that $\sigma^k \sim \id$ or $\sigma^k\varphi = \varphi$ holds. Then
  $\im(\varphi) \supseteq \im(\sigma\varphi) \supseteq \ldots \supseteq \im(\sigma^k\varphi) = \im(\varphi)$.
  Clearly, there exists a positive integer $0<j\leq k$ minimal subject to the existence
  of $0\leq i<j$ so that $\sigma^j \leadsto_\varphi \sigma^i$. Hence,
  there exists a homomorphism $\pi\colon \im(\sigma^i\varphi)
  \to \im(\sigma^j\varphi)$ so that $\sigma^j\varphi =
  \sigma^i\varphi\pi$. Note that $\pi$ is surjective. As $k-i>0$, we
  have $\sigma^{k-i}\sigma^j\varphi = \sigma^{k-i}\sigma^i \varphi
  \pi = \sigma^k\varphi\pi = \varphi\pi$. On the other hand, we
  have $\sigma^{k-i}\sigma^j\varphi = \sigma^{j-i}\sigma^k\varphi =
  \sigma^{j-i}\varphi$. Hence $\sigma^{j-i}\varphi = \varphi\pi$. If $i >
  0$, the latter contradicts the minimality of $j$. Thus $i=0$ and we have
  $\sigma^j\varphi = \varphi\pi$ for a homomorphism $\pi\colon\im(\varphi)
  \to \im(\sigma^j\varphi)$. Since $\im(\varphi) \supseteq
  \im(\sigma\varphi) \supseteq  \ldots \supseteq \im(\sigma^k\varphi) =
  \im(\varphi)$, the homomorphism $\pi$ is an automorphism
  of the finite group $\im(\varphi)$. As $\im(\varphi)$ is finite, the automorphism $\pi$
  has finite order $n$, say. Suppose that $nj < k$ holds. Then we can
  write $k = s\cdot nj + t$ with $0\leq t < nj$ and $s \in\N$. This yields
  that $\varphi = \sigma^{k}\varphi = \sigma^t\,\sigma^{s\,nj}\varphi
  = \sigma^t \varphi (\pi^n)^s = \sigma^t \varphi$ and $\sigma^t
  \leadsto_\varphi \id$. By the minimality of $j$, we have $t \geq
  j$. Therefore, we can write $t = m\,j+\ell$ with $0\leq \ell < j$ and
  $m\in\N$. This yields that $\varphi = \sigma^t \varphi = \sigma^\ell
  \varphi \pi^m$ and thus $\sigma^\ell \leadsto_\varphi \id$. If $\ell
  > 0$, then $\sigma^\ell \leadsto_\varphi \id$ contradicts the
  minimality of $j$. Thus $t = mj$ and $j \mid k$ because $k = (sn+m)j$.
  This yields that $\varphi = \sigma^k \varphi = \sigma^{(sn+m)j}\varphi =
  \varphi\pi^{sn+m}$ and, as $n$ is the order of the automorphism $\pi$,
  we obtain $n \mid sn+m$ and $nj \mid k$. If, on the other hand, $nj > k$
  holds, then $j \leq k < nj$ and we can write $k = mj + \ell$ with 
  $0\leq \ell < j$ and $m\in\N$. Then $\varphi = \sigma^k \varphi = 
  \sigma^{mj+\ell}\varphi = \sigma^\ell \varphi \pi^m$ and so 
  $\sigma^\ell \leadsto_\varphi \id$. The minimality of $j$ yields 
  $\ell = 0$ as above and hence $k = mj$. Moreover, we have 
  $\varphi = \sigma^k\varphi = \sigma^{mj}\varphi = \varphi \pi^m$ 
  and thus the order $n$ of the automorphism $\pi$ divides the 
  integer $m$; in particular, we obtain $nj \mid mj = k$ which contradicts 
  the assumption $k < nj$. Write $\ell = nj$. If $\Phi = \{\sigma\}$, then the set 
  $\{\id,\sigma,\ldots,\sigma^{\ell-1}\}$ is an upper bound on the set $\ti V$ from 
  Lemma~\ref{lem:LeadstoSet}
  because $\sigma^\ell \leadsto_\varphi \id$ holds. By the minimal choice of 
  $\ell$, we obtain that $\ti V = \{\id,\sigma,\ldots,\sigma^{\ell-1}\}$.\smallskip

  Suppose that both $\sigma^\ell \leadsto_\varphi
  \id$ and $\id \leadsto_\varphi \sigma^\ell$ hold. Then, as we already proved above, 
  there exists an isomorphism $\pi\colon \im(\varphi) \to
  \im(\sigma^\ell\varphi)$ with $\sigma^\ell\varphi = \varphi\pi$. Since $\im(\sigma^\ell\varphi) 
  \subseteq \im(\varphi)$ and $\pi$ is bijective, $\pi$ is an automorphism of $\im(\varphi)$. Then
  automorphism $\pi$ of the finite group $\im(\varphi)$, 
  has finite order $n$, say. Write $k = n\ell$. Then
  $\sigma^k\varphi = \sigma^{n\ell}\varphi = \varphi\pi^n = \varphi$
  and so $\sigma^k \sim \id$. Suppose that $\Phi = \{\sigma\}$ and that the integer 
  $\ell>0$ above is minimal. Then, by our minimal choice of $k$, we obtain 
  $V = \{\id,\sigma,\ldots,\sigma^{k-1}\}$ for the set $V$ returned by 
  Algorithm~\ref{alg:IteratingEndomorphisms}.\smallskip

  The permutation representation $\varphi\colon F\to \Sym(\UK\bs F)$ of the subgroup $\U =
  \la a,b^2,bab^{-1} \ra$ is induced by the map $a\mapsto (\:)$ and $b
  \mapsto (1,2)$. Therefore, $\U$ satisfies that $\sigma \leadsto_\varphi
  \id$ and $|\im(\varphi)| = 2$ though $|\im(\sigma\varphi)| = 1$. In
  particular, for each $\ell \geq 1$, we have $|\im(\sigma^\ell\varphi)|
  = 1$ and thus there is no integer $\ell$ so that $\sigma^\ell \sim
  \id$ holds. However, we have $\sigma^2\varphi = \sigma\varphi$ so that the 
  set  $V = \{\id,\sigma,\sigma^2\}$ returned by Algorithm~\ref{alg:IteratingEndomorphisms} 
  is finite.
\end{proof}
The stabilizing core $\El$ introduced in Definition~\ref{def:StabSubs}
satisfies the following
\begin{proposition}\Label{prop:StabCore}
  Let $V\subseteq \Phi^*$ be the finite set returned by
  Algorithm~\ref{alg:IteratingEndomorphisms}.  The stabilizing core $\El$
  satisfies that
  \[
    \El = \bigcap_{\sigma\in V} \ker(\sigma\varphi). 
  \]
  Moreover, $\El$ is the largest $\Phi$-invariant subgroup of $\UK$
  which is normal in $F$ and thus $\El = \Core_F(\ti\El)$. It is finitely
  generated, it has finite index in $F$, and it contains all iterated
  relations $\R$ of the $L$-presentation $\la\X\mid\Q\mid\Phi\mid\R\ra$
  of $G$. We have $\El \subseteq \ti\El \subseteq \UK \subseteq F$
  and $\El \subseteq \Core_F(\UK) \subseteq \UK \subseteq F$.
\end{proposition}
\begin{proof}
  Write ${\mc K} = \bigcap_{\sigma\in V} \ker(\sigma\varphi)$. Clearly
  $\El \subseteq {\mc K}$. Let $g \in {\mc K}$ be given. Then,
  for all $t \in T$, we have $\UK\,t\cdot g^\sigma = \UK\,t$ for each
  $\sigma\in V$. Let $\delta \in \Phi^*$ be given. By Lemma~\ref{lem:Term}, there exists
  $\sigma \in V$ with $\delta\varphi = \sigma\varphi$. Thus $\UK\,t \cdot
  g^\delta = \UK\,t\cdot g^\sigma = \UK\,t$ for each $t\in T$. Hence $g^\delta$ stabilizes 
  all right-cosets $\UK\,t$ and thus $g \in
  \ker(\delta\varphi)$. As $\delta\in\Phi^*$ was arbitrarily chosen, we have
  $\El = {\mc K}$.\smallskip

  The stabilizing core $\El$ is normal in $F$ because it is the intersection
  of normal subgroups. Since $\El \subseteq \ker( \varphi ) = \Core_F(\UK)$ holds, the stabilizing core $\El$ is contained
  in $\UK$. As $\El = \bigcap_{\sigma\in\Phi^*} \ker(\sigma\varphi)$ holds, the subgroup $\El$ is 
  $\Phi$-invariant.
  Let $N$ be a $\Phi$-invariant
  subgroup which is normal in $F$ and which satisfies $\El \leq N \leq
  \UK$. Let $g\in N$, $t\in T$, and $\sigma \in V$ be given.  Since $N$
  is $\Phi$-invariant, we have $g^\sigma \in N$. As $N \unlhd F$
  we also have $tg^\sigma t^{-1} \in N$ or $tg^\sigma = vt$ for some $v \in
  N\subseteq \UK$. Thus $\UK\,t\cdot g^\sigma = \UK\,vt = \UK\,t$ and so
  $g \in \ker(\sigma\varphi)$. As $\sigma\in V$ was arbitrarily chosen, we have $g\in \El$. This yields that $N \subseteq \El$ and 
  hence $N=\El$.\smallskip

  The stabilizing core $\El$ has finite index in $F$ because it is
  the intersection of finitely many finite index subgroups $\ker(\sigma\varphi)$ with $\sigma \in V$. Moreover,
  $\El$ is finitely generated as a finite index subgroup of a finitely
  generated free group $F$. Let $r\in\R$ be an iterated relator of
  the $L$-presentation $\la \X\mid\Q\mid\Phi\mid\R\ra$ of $G$. Then,
  for each $\sigma \in V$, the image $r^\sigma$ is a relator of $G$ as well and
  thus we have $r\in\ker( \sigma\varphi )$ and so $r \in \El$.\smallskip

  As $\El$ is $\Phi$-invariant, we have $\El \subseteq\ti\El$. Since
  $\El$ is normal in $F$ and a subgroup of $\UK$, we have $\El\subseteq
  \Core_F(\UK)$.
\end{proof}
Because the stabilizing core $\El$ contains the iterated relations $\R$ of the $L$-presentation,
the normal closure $\la\bigcap_{\sigma\in\Phi^*} \R^\sigma\ra^F$
is contained in $\El$ as well. This yields the immediate
\begin{corollary}\Label{cor:InvLpFinExt}
  If $G = \la\X\mid\Q\mid\Phi\mid\R\ra = \la\X\mid\emptyset\mid\Phi\mid\Q\cup\R\ra$ is invariantly 
  $L$-presented so that $G = F / K$, we have $K \subseteq \El \subseteq \ti\El \subseteq
  \UK \subseteq F$. Hence, the subgroup $\U \cong \UK/K \leq F/K = G$ contains
  the $\Phi$-invariant normal subgroup $\El/K$. The index $[\UK/K:\El/K] = [\UK:\El]$ 
  is finite.
\end{corollary}
Whence the subgroup $\U$ in Corollary~\ref{cor:InvLpFinExt} is a finite extension of $\El/K$.  Since $\El$
is the largest $\Phi$-invariant subgroup which is normal in $F$, the
stabilizing subgroup $\ti\El$ is normal in $F$ if and only
if $\El = \ti\El$ holds. Moreover, we have the following
\begin{lemma}
  We have $\ti\El = \El$ if and only if $\ti\El \subseteq \Core_F(\UK)$ holds.
\end{lemma}
\begin{proof}
  We have $\El \subseteq \ti\El$ and 
  $\ti\El^\psi \subseteq\ti\El$, for each $\psi \in \Phi$.
  If $\El = \ti\El$, then $\ti\El = \El \subseteq 
  \Core_F(\UK)$. On the other hand, suppose that $\ti\El 
  \subseteq \Core_F(\UK)$ holds. Let $g \in\ti\El$ and 
  $\sigma\in V$ be given. Then $g^\sigma \in \ti\El$, as 
  $\ti\El$ is $\sigma$-invariant. Since $g^\sigma \in
  \ti\El \subseteq \Core_F(\UK)$ holds, we have
  $tg^\sigma\,t^{-1} \in \Core_F(\UK) \subseteq \UK$, for each $t\in T$. This 
  yields that $\UK\,t\cdot g^{\sigma} = \UK\,t$ and thus $g^\sigma$ acts
  trivially on the right-cosets $\UK\bs F$. In particular, we have $g^\sigma \in \ker(\varphi)$
  and $g \in \ker(\sigma\varphi)$. As $\sigma \in V$ was arbitrarily chosen,
  we have $g \in \El = \bigcap_{\sigma\in V} \ker(\sigma\varphi)$.
\end{proof}
If $\UK \unlhd F$ is a normal subgroup, then $\ti\El \subseteq
\UK = \Core_F(\UK)$ holds and hence, we obtain the immediate
\begin{corollary}\Label{cor:StabEq}
  If $\UK \unlhd F$, then $\El = \ti\El$.
\end{corollary}
Note the following
\begin{remark}
  There are subgroups that satisfy $\Core_F(\UK) \subset\ti\El$. For
  instance, the subgroup $\U = \la a, b^2, ba^2b^{-1}, bab^{-2}a^{-1}b^{-1} \ra$
  of the Basilica group is $\Phi$-in\-vari\-ant (and hence $\ti\El = \UK$ 
  by Lemma~\ref{lem:StabSubInv}) but not normal in $G$.\smallskip

  There are subgroups that satisfy $\ti\El \subset \Core_F(\UK)$. For
  instance, the subgroup $\U = \la a^2, b, aba^{-1} \ra$ of the Basilica
  group has index $2$ in $G$ (and thus it is normal in $G$); though the subgroup
  $\U$ is not $\sigma$-invariant.\smallskip

  There are subgroups that neither satisfy $\ti\El \subseteq \Core_F(\UK)$
  nor $\Core_F(\UK) \subseteq \ti\El$. For instance, the subgroup 
  $\U = \la a, bab^{-1}, b^{-1}a^2b, b^2ab^2, b^3a^{-1}b\ra$ of the Basilica group 
  satisfies $[F:\ti\El] = [F:\Core_F(\UK)]$ and $\ti\El \neq \Core_F(\UK)$.
\end{remark}

\section{The Reidemeister-Schreier theorem}\Label{sec:ReidSchrThm}
In this section, we finally prove our variant of the Reidemeister-Schreier theorem in
Theorem~\ref{thm:CentralThm}. For this purpose, let $G = \la \X \mid \Q
\mid \Phi \mid \R\ra$ be a finitely $L$-presented group and let $\U \leq G$ be
a finite index subgroup given by its generators $g_1,\ldots,g_n$, say. We
consider the generators $g_1,\ldots,g_n$ as elements of the free
group $F$ over $\X$.  Denote the normal closure of the relations
of $G$ by $K = \la \Q\cup\bigcup_{\sigma\in\Phi^*} \R^\sigma \ra^F$
and let $U = \la g_1,\ldots,g_n\ra \leq F$.  Then $\U \cong \UK/K$.
If $T\subseteq F$ denotes a Schreier transversal for $\UK$ in $F$,
the Reidemeister-Schreier Theorem in Section~\ref{sec:ReidSchr} shows that
the subgroup $\U$ admits the group presentation
\begin{equation}\Label{eqn:SubgrpLPres}
  \U \cong\Big\la\,{\mc Y} \:\Big|\:
  \{\tau(tqt^{-1}) \mid t\in T,q\in\Q\} \cup \bigcup_{\sigma\in\Phi^*}
  \{\tau(tr^{\sigma}t^{-1}) \mid t\in T,r\in\R\}\Big\ra,
\end{equation}
where $\tau$ is the Reidemeister rewriting.
We will construct a finite $L$-presentation
from the group presentation in Eq.~(\ref{eqn:SubgrpLPres}). First, we 
note the following
\begin{theorem}\Label{thm:NormInvSubgrps}
  Let $G = \la\X\mid\Q\mid\Phi\mid\R\ra$ be invariantly finitely
  $L$-presented. Each $\Phi$-invariant normal subgroup with finite index in $G$ 
  is invariantly $L$-presented.
\end{theorem}
\begin{proof}
  Let $G = \la\X\mid\Q\mid\Phi\mid\R\ra$ be an invariantly finitely
  $L$-presented group and let $\U\unlhd G$ be a $\Phi$-invariant normal
  subgroup with finite index in $G$.  Every invariantly $L$-presented
  group can be considered as an ascendingly $L$-presented group by
  Proposition~\ref{prop:AscInvLp}. Therefore, we may consider $\Q =
  \emptyset$ in the following. Consider the notation introduced above.
  As $G$ is invariantly $L$-presented, we have $K^\sigma \subseteq K$ for
  each $\sigma \in \Phi^*$. Since the subgroup $\U$ is $\Phi$-invariant,
  we also have $U^\sigma \subseteq U$ and therefore $(\UK)^\sigma \subseteq
  \UK$ for each $\sigma \in \Phi^*$. By 
  Lemma~\ref{lem:StabSubInv}, we have $\ti\El = \UK$. Furthermore,
  as $\UK \unlhd F$ holds, we have $\El = \ti\El$ and thus $\UK =
  \ti\El = \El$. Let $t\in T$ be given. As $\U\unlhd G$ holds, the
  mapping $\delta_t\colon \UK \to \UK,\: g \mapsto tgt^{-1}$ defines
  an automorphism of $\UK$. The Reidemeister rewriting $\tau\colon \UK
  \to F({\mc Y})$ is an isomorphism of free groups and therefore the
  endomorphisms $\Phi \cup \{\delta_t \mid t\in T\}$ of $\UK$ translate
  to endomorphisms $\widehat\Phi \cup \{\widehat\delta_t \mid t \in
  T\}$ of the free group $F({\mc Y})$. Consider the invariant finite
  $L$-presentation
  \begin{equation}\label{eqn:SubgrpInvLp}
    \la\,{\mc Y} \mid \emptyset \mid \widehat\Phi \cup \{\widehat\delta_t \mid t \in T \} \mid \{\tau(r)\mid r\in\R\}\,\ra.
  \end{equation}
  In order to prove that the finite $L$-presentation in
  Eq.~(\ref{eqn:SubgrpInvLp}) defines the subgroup $\U$, it
  suffices to prove that each relation of the presentation in
  Eq.~(\ref{eqn:SubgrpLPres}) is a consequence of the relations of
  the $L$-presentation in Eq.~(\ref{eqn:SubgrpInvLp}) and vice versa.
  For $t\in T$, $r\in\R$, and $\sigma\in\Phi^*$, we consider the
  relation $\tau(t\,r^\sigma\,t^{-1})$ of the group presentation in
  Eq.~(\ref{eqn:SubgrpLPres}). Clearly, this relation is contained in the
  finite $L$-presentation in Eq.~(\ref{eqn:SubgrpInvLp}) as there exists
  $\widehat \sigma \in \widehat\Phi^*$ so that $(\tau(r))^{\widehat\sigma}
  = \tau(r^\sigma)$. Then $(\tau(r))^{\widehat\sigma\delta_t} =
  \tau(tr^\sigma t^{-1})$. On the other hand, consider the relation
  $\tau(r)^{\widehat\sigma}$ of the finite $L$-presentation in
  Eq.~(\ref{eqn:SubgrpInvLp}) where $r\in\R$ and $\widehat\sigma \in
  (\widehat\Phi\cup\{\widehat\delta_t\mid t\in T\})^*$. Write $\Psi =
  \widehat\Phi \cup \{\widehat\delta_t \mid t \in T\}$. Since $1\in T$ and
  $\id\in \Phi^*$, we can write each image of an element $\widehat\delta
  \in \Psi$ as $\tau(g)^{\widehat\delta} = \tau(tg^\delta\,t^{-1})$
  for some $t \in T$ and $\delta\in \Phi^*$. Since $\widehat\sigma\in
  \Psi^*$, we can write $\widehat \sigma = \widehat \sigma_1 \cdots
  \widehat \sigma_n$ with each $\widehat \sigma_i \in \Psi$. Then the
  image $\tau(r)^{\widehat\sigma}$ has the form
  \[
    \tau(r)^{\widehat\sigma}
    = \tau( t_n \cdots t_2^{\sigma_3 \cdots \sigma_n}\,t_1^{\sigma_2\sigma_3\cdots \sigma_n}\cdot r^{\sigma_1\sigma_2\cdots\sigma_n}\cdot t_1^{-\sigma_2\sigma_3\cdots\sigma_n}\,t_2^{-\sigma_3\cdots\sigma_n}\cdots t_n^{-1}).
  \]
  Since $T$ is a transversal for $\UK$ in $F$, we can write $t_n
  \cdots t_2^{\sigma_3 \cdots \sigma_n}\,t_1^{\sigma_2\sigma_3\cdots
  \sigma_n} = u\,t$ where $t\in T$ and $u\in \UK$. This
  yields that $\tau(r)^{\widehat\sigma} = \tau(
  u\,t\,r^{\sigma_1\sigma_2\cdots\sigma_n}\,t^{-1}\,u^{-1} ) =
  \tau(u)\,\tau(t\,r^{\sigma_1\sigma_2\cdots\sigma_n}\,t^{-1})\,\tau(u)^{-1}$,
  which is a consequence of
  $\tau(t\,r^{\sigma_1\sigma_2\cdots\sigma_n}\,t^{-1})$. 
  The latter relation $\tau(t\,r^{\sigma_1\sigma_2\cdots\sigma_n}\,t^{-1})$ is a relation of the group presentation in
  Eq.~(\ref{eqn:SubgrpLPres}).  In summary, each relation of the group
  presentation in Eq.~(\ref{eqn:SubgrpLPres}) is a consequence of the
  finite $L$-presentation in Eq.~(\ref{eqn:SubgrpInvLp}) and vice versa.
\end{proof}
In order to prove our Reidemeister-Schreier theorem for finitely
$L$-presented groups, we need to consider finite index subgroups
that are not normal.  For this purpose, we need to construct the
relations $\tau(tr^\sigma\,t^{-1})$, with $t \in T$, $r \in \R$, and
$\sigma\in \Phi$. The overall strategy in this paper is to construct
the relations as iterated images of the form $\tau( sr\,s^{-1}) ^
{\widehat\sigma}$ for $s\in T$ and some $\widehat\sigma \in \widehat\Phi^*$. If
the subgroup $\U$ is normal as in Proposition~\ref{prop:ReidSchrInv},
the conjugation action $\delta_t\colon \UK \to \UK$ enables us to first
construct the image $\tau(r^\sigma) = \tau(r)^{\widehat\sigma}$ and
then to consider the conjugates $\tau(r^{\sigma})^{\widehat\delta_t}
= \tau(tr^{\sigma}t^{-1})$. However, in general, it is not sufficient
to take as iterated relations those $\tau(trt^{-1})^\sigma = \tau(
t^\sigma r^\sigma t^{-\sigma})$, with $t\in T$ and $r\in\R$, as $\sigma$ may not be invertible over
$\{ t^\sigma \mid t\in T\}$. More precisely, we have the following
\begin{remark} 
  Let $\U = \la a, b^2, ba^3b^{-1}, bab^{-2}a^{-1}b^{-1},
  ba^{-1}b^{-2}ab^{-1}\ra$ be a subgroup of the Basilica group $G$.
  The subgroup $\U$ is $\sigma$-invariant and thus we can consider
  the iterated images $\{ \tau(r)^{\widehat\sigma} \mid r\in
  \R,\sigma\in\Phi^*\}$. A Schreier transversal $T$ for $\U$ in $G$
  is given by $T = \{ 1, b, ba, ba^2, bab, ba^2b \}$. We have $T ^
  \sigma = \{ 1, a, ab^2, ab^4, ab^2a, ab^4a \}$.  Note that $T^\sigma
  \subseteq \UK$ holds. Thus we cannot ensure that the iterated images
  $\{\tau(trt^{-1})^{\widehat\sigma} \mid r\in\R,t\in T,\sigma\in\Phi^*\}$
  contain all relations in Eq.~(\ref{eqn:SubgrpLPres}). As the subgroup $\U$ is not
  normal in $G$, we cannot consider the conjugate action as well. However, an
  invariant finite $L$-presentation for the subgroup $\U$ can be computed
  with Theorem~\ref{thm:LeafInv} as the subgroup $\U$ is leaf-invariant (see Section~\ref{sec:InvSubgrpLpres} below).
\end{remark}
In the following, we use Theorem~\ref{thm:NormInvSubgrps} to prove our
variant of the Reidemeister-Schreier Theorem for invariantly
finitely $L$-presented
groups first.
\begin{proposition}\Label{prop:ReidSchrInv}
  Every finite index subgroup of an invariantly finitely 
  $L$-pre\-sen\-ted group is finitely $L$-presented. 
\end{proposition}
\begin{proof}
  Let $\U$ be a finite index subgroup of an invariantly finitely
  $L$-presented group $G = F/K$. By Corollary~\ref{cor:InvLpFinExt},
  the subgroup $\U \cong \UK/K$ contains a normal subgroup $\El / K$
  with finite index in $G$ and which is $\Phi$-invariant. By
  Theorem~\ref{thm:NormInvSubgrps}, the subgroup $\El/K \leq F/K$ is
  finitely $L$-presented. The subgroup $\U$ is a finite extension
  of a finitely $L$-presented group and thus, by Corollary~\ref{cor:FinExt},
  the subgroup $\U$ is finitely $L$-presented itself.
\end{proof}
Recall that we do not have a method to construct an invariant
$L$-presentation for a finite extension of an invariantly $L$-presented
group. Therefore, we cannot ensure invariance of the finite
$L$-presentation obtained from Corollary~\ref{cor:InvLpFinExt}.  We will
study in Section~\ref{sec:InvSubgrpLpres} conditions on a subgroup
of an invariantly $L$-presented group that ensure the invariance
of the subgroup $L$-presentation.  First, we complete our proof of
Theorem~\ref{thm:CentralThm}:
\def\0{Theorem~\ref{thm:CentralThm}}
\begin{CenProof}{ of \0}
  Let $G = \la\X\mid\Q\mid\Phi\mid\R\ra$ be a finitely $L$-presented
  group and let $\U$ be a finite index subgroup of $G$. Denote the
  free group over $\X$ by $F$. Define the normal subgroups $K =
  \la \Q \cup \bigcap_{\sigma\in\Phi^*} \R^\sigma\ra^F$ and $L = \la
  \bigcap_{\sigma\in\Phi^*} \R^\sigma\ra^F$. Let $U \leq F$ be generated
  by the generators of $\U$ so that $\U \cong \UK/K$ holds. Then we have
  $L \unlhd K \unlhd F$ and $G = F/K$. Further, the group $H = F/L$ is
  invariantly $L$-presented by $\la\X\mid\emptyset\mid \Phi\mid\R\ra$ and 
  it naturally maps onto $G$.
  The subgroup $\UK/L \leq F/L$ has finite index in $H$ as $[F:\UK]$
  is finite. By Proposition~\ref{prop:ReidSchrInv}, the subgroup $\UK/L$
  of the invariantly finitely $L$-presented group $H = F/L$ is finitely
  $L$-presented.  The exact sequence $1\to K/L\to \UK/L \to \UK/K \to 1$
  yields that $\U \cong \UK/K \cong (\UK/L) / (K/L)$ where the kernel
  $K/L$ is finitely generated, as a normal subgroup, by the image of the fixed 
  relations in $\Q$. Thus, by Proposition~\ref{prop:Factor}, $\U$ is finitely $L$-presented as a factor 
  group of a finitely $L$-presented group whose kernel is finitely 
  generated as a normal subgroup.
\end{CenProof}

\section{Invariant subgroup $L$-presentations}\Label{sec:InvSubgrpLpres}
The algorithms in~\cite{BEH08,Har10a} are much more efficient on
invariant $L$-presentations. Therefore, we will study conditions on
the subgroup $\U$ of an invariantly $L$-presented group $G$ to be
invariantly $L$-presented itself. By Theorem~\ref{thm:NormInvSubgrps}, each
$\Phi$-invariant normal subgroup $\U$ of an invariantly finitely $L$-presented
group $G = \la \X\mid\Q\mid\Phi\mid\R\ra$ is invariantly finitely $L$-presented
as soon as $[G:\U]$ is finite.\smallskip

Let $\varphi\colon F \to \Sym(\UK\bs F)$ be a permutation representation
as usual. 
Recall that the subgroup $\U$ is leaf-invariant, if the $\varphi$-leafs
\[
  \Psi = \{ \psi\delta \mid \psi\in\Phi, \delta\in V,
            \psi\delta \not\in V, \psi\delta\varphi = \varphi\},
\]
of $V$ satisfy $\Psi = \{\psi\delta \mid \psi\in\Phi, \delta\in
V,\psi\delta\not\in V \}$; cf. Definition~\ref{def:LeafInv}. This definition yields
the following
\begin{theorem}\Label{thm:LeafInv}
  Each leaf-invariant, finite index subgroup of an invariantly finitely
  $L$-presented group is invariantly finitely $L$-presented.
\end{theorem}
\begin{proof}
  Let $G=\la\X\mid\Q\mid\Phi\mid\R\ra$ be invariantly finitely
  $L$-presented and let $\U \leq G$ be a leaf-invariant finite index
  subgroup of $G$. Clearly, we can consider $\Q = \emptyset$ in the following. The $\varphi$-leafs $\Psi$ satisfy $\Psi =
  \{\psi\delta \mid \psi\in\Phi, \delta\in V,\psi\delta\not\in V \}$.
  By Lemma~\ref{lem:AscLpInv},
  each $\varphi$-leaf $\psi\delta\in \Psi \subseteq \Phi^*$ defines an endomorphism
  of the subgroup $\UK$. Moreover, Lemma~\ref{lem:AscLpInv} shows
  that each $\sigma\in \Phi^*$ can be written as $\sigma =
  \vartheta\,\delta$ with $\vartheta\in V$ and $\delta
  \in \Psi^*$. Consider the finite $L$-presentation 
  \begin{equation}\Label{eqn:InvSubgrpPresProof}
    \la {\mc Y} \mid \emptyset \mid \{\widehat{\psi\delta}\mid \psi\delta\in \Psi\}\mid 
    \{ \tau(tr^\vartheta\,t^{-1}) \mid \vartheta \in V, r \in \R, t\in T\} \ra,
  \end{equation}
  where ${\mc Y}$ denotes the Schreier generators of $\UK$ and 
  $\widehat{\psi\sigma}$ denotes the endomorphism of the free group $F({\mc Y})$ 
  induced by the endomorphisms $\psi\sigma$ of $\UK$.
  For $t\in T$, $\sigma\in\Phi^*$, and $r\in\R$, the relation
  $\tau(t\,r^{\sigma}\,t^{-1})$ of the group presentation in
  Eq.~(\ref{eqn:SubgrpLPres}) can be obtain from the above
  $L$-presentation as follows:  Since each $\sigma\in\Phi^*$ can
  be written as $\sigma = \vartheta\,\delta$ with $\vartheta
  \in V$ and $\delta\in\Psi^*$, we claim that the
  relation $\tau(t\,r^{\sigma}\,t^{-1})$ is a consequence of
  the image $\tau(tr^\vartheta\,t^{-1})^{\widehat\delta}$. The latter
  image satisfies that $\tau(tr^{\vartheta}\,t^{-1})^{\widehat\delta}
  = \tau( t^\delta\,r^{\vartheta\delta}\,t^{-\delta}) = \tau(
  t^\delta\,r^{\sigma}\,t^{-\delta})$. As $\delta\in \Psi^*$, we can write
  $\delta = \delta_1\cdots\delta_n$  with each $\delta_i\in\Psi$. Recall
  that $\delta_i\varphi = \varphi$ holds. Thus the right-coset $\UK\,1$
  satisfies that $\UK\,1\cdot t^{\sigma_i} = \UK\,1\cdot t = \UK\,t$ and therefore
  $\UK\,t^{\delta_1\cdots\delta_n} = \UK\,t$. Hence, there exists $u \in
  \UK$ so that $t^\delta = ut$ and we obtain
  \[
    \tau(tr^{\vartheta}\,t^{-1})^{\widehat\delta} = \tau( t^\delta\,r^{\sigma}\,t^{-\delta})
    = \tau( ut\,r^\sigma\,t^{-1}\,u^{-1} ) 
    = \tau(u)\,\tau(t\,r^\sigma\,t^{-1})\,\tau(u)^{-1}
  \]
  which is a consequence of $\tau(t\,r^\sigma\,t^{-1})$ and vice
  versa.  Similarly, every relation of the $L$-presentation in
  Eq.~(\ref{eqn:InvSubgrpPresProof}) is a consequence of the relations
  in Eq.~(\ref{eqn:SubgrpLPres}). Therefore, the invariant finite
  $L$-presentation in Eq.~(\ref{eqn:InvSubgrpPresProof}) defines the
  leaf-invariant finite index subgroup $\U$.
\end{proof}
For finite $L$-presentations $\la\X\mid\Q\mid\Phi\mid\R\ra$ with $\Phi = \{\sigma\}$, the leaf-invariance of the
subgroup $\U$ yields the existence of a positive integer $j$
so that $\sigma^j\varphi = \varphi$ holds. If we assume the positive integer $j$
to be minimal, then $V = \{\id,\sigma,\ldots,\sigma^{j-1}\}$ and $\Psi =
\{\sigma^j\}$. In this case, the invariant finite $L$-presentation in
Eq.~(\ref{eqn:InvSubgrpPresProof}) becomes
\[
  \U \cong \la {\mc Y} \mid \emptyset \mid \{\widehat{\sigma^j}\} \mid 
  \{ \tau(tr^{\sigma^i}\,t^{-1}) \mid t\in T, r\in\R, 0\leq i < j \} \ra.
\]
Note that the subgroup $\U$ in Theorem~\ref{thm:LeafInv} is not
necessarily normal in $G$. However, leaf-invariance of a subgroup is a
restrictive condition on the subgroup.  We try to weaken this condition with
the following
\begin{definition}\Label{def:WeakInv}
  Let $G=\la\X\mid\Q\mid\Phi\mid\R\ra$ be a finitely $L$-presented group and let 
  $\U \leq G$ be a finite index subgroup with permutation representation $\varphi$. 
  Then the subgroup $\U$ is \emph{weakly leaf-invariant}, if 
  \[
    \Psi = \{\psi\delta \mid \psi\in\Phi, \delta \in V, \psi\delta \not\in V, \psi\delta \leadsto_\varphi \id \}
  \]
  satisfies $\Psi = \{\psi\delta \mid \psi\in\Phi, \delta \in V, \psi\delta \not\in V\}$.
\end{definition}

The notion of a weakly leaf-invariant subgroup is less restrictive than
leaf-invariance, as low-index subgroups of the Basilica groups suggest:
Among the $4\,956$ low-index subgroups of the Basilica group with index
at most $20$ there are $2\,539$ weakly leaf-invariant subgroups; only
$156$ of these subgroups are leaf-invariant. More precisely, Table~\ref{tab:SubgrpCnts}
shows the number of subgroups ($\leq$) that are
normal ($\unlhd)$, maximal (max), 
leaf-invariant (l.i.), weakly leaf-invariant 
(w.l.i.), and the number of subgroups that are weakly leaf-invariant 
and normal ($\unlhd + {\rm w.l.i.}$).
\begin{table}[h]
  \caption{Subgroups of the Basilica group with index at most $20$.}
  \label{tab:SubgrpCnts}
  \[
    \begin{array}{ccccccc}
    \toprule
    {\rm index} & \leq & \unlhd & {\rm max} & {\rm l.i.} & {\rm w.l.i} & \unlhd + {\rm w.l.i} \\
    \midrule
    1 & 1  & 1 & 1 & 1 & 1 & 1 \\ 
    2 & 3  & 3 & 3 & 0 & 3 & 3 \\ 
    3 & 7  & 4 & 7 & 4 & 4 & 4 \\ 
    4 &19  & 7 & 0 & 0 &19 & 7 \\ 
    5 &11  & 6 &11 & 6 & 6 & 6 \\ 
    6 &39  &13 & 0 & 0 &14 & 12\\ 
    7 &15  & 8 &15 & 8 & 8 & 8 \\ 
    8 &163 &19 & 0 & 0 &139& 19\\ 
    9 &115 &13 & 9 &49 &52 & 13\\ 
   10 &83  &19 & 0 & 0 &22 & 18\\ 
   11 &23  &12 &23 &12 &12 & 12\\ 
   12 &355 &31 & 0 & 0 &98 & 28\\ 
   13 &27  &14 &27 &14 &14 & 14\\ 
   14 &115 &25 & 0 & 0 &30 & 24\\ 
   15 &77  &24 & 0 &24 &24 & 24\\ 
   16 &1843&47 & 0 & 0 &1531& 43\\ 
   17 &35  &18 &35 &18 &18 & 18\\ 
   18 &1047&44 & 0 & 0 &366& 40\\ 
   19 &39  &20 &39 &20 &20 & 20\\ 
   20 &939 &45 & 0 & 0 &158 & 42\\ 
    \bottomrule
    \end{array}
  \]
\end{table}
For finite $L$-presentations $\la\X\mid\Q\mid\Phi\mid\R\ra$ with $\Phi
= \{\sigma\}$, each leaf-invariant subgroup is weakly leaf-invariant
by Lemma~\ref{lem:SimLeadsto}, (iii). On the other hand, a weakly
leaf-invariant subgroup with $\Phi = \{\sigma\}$ such that both $\sigma^\ell
\leadsto_\varphi \id$ and $\id \leadsto_\varphi \sigma^\ell$ hold, is
leaf-invariant by Lemma~\ref{lem:SimLeadsto}, (iv).  There are subgroups
of a finitely $L$-presented group that are weakly leaf-invariant but
not leaf-invariant; see Lemma~\ref{lem:SimLeadsto}, (v).
If $\Phi$ contains more than one generator, then we may ask the following
\begin{question}
  Is every leaf-invariant subgroup weakly leaf-invariant?
\end{question}
The problem is that Definitions~\ref{def:LeafInv} and~\ref{def:WeakInv}
depend on the minimal sets $V$ and $\ti V$ which satisfy $\ti V\subseteq
V$ but which may differ in general. We do not have an answer to this
question.

Moreover the sets $V$ and $\ti V$ in the Definitions~\ref{def:LeafInv} and~\ref{def:WeakInv}
may also depend on the ordering $\prec$ chosen in our
Algorithm~\ref{alg:IteratingEndomorphisms}. Though we have the following
\begin{lemma}
  The conditions leaf-invariance and weak leaf-invariance do not depend 
  on the choice of the ordering $\prec$ in Algorithm~\ref{alg:IteratingEndomorphisms}.
\end{lemma}
\begin{proof}
  We show the claim for the weaker condition of weak
  leaf-invariance and we show this by proving that the set $\ti V$
  in Lemma~\ref{lem:LeadstoSet} does not depend on the ordering. Suppose that 
  a subgroup $\U$ is weakly leaf-invariant with respect to the ordering $\prec$. Let
  $\ti V_\prec$ and $\ti V_<$ be the sets with respect to the orderings
  $\prec$ and $<$, respectively. We first show that
  $V_< \subseteq V_\prec$ holds. Let $\sigma \in V_<$ be a $\prec$-minimal
  counter-example with $\sigma \not \in V_\prec$. As $\id \in V_\prec$,
  we have $\sigma \neq \id$ and therefore we can write $\sigma =
  \psi\delta$ with $\psi\in\Phi$ and $\delta \in \Phi^*$. Now, $V_<$
  can be considered as a subtree and hence, we have $\delta \in V_<$ and
  $\delta \prec \sigma$. By the minimality of $\sigma$, we have $\delta
  \in V_\prec$. Thus the element $\sigma = \psi\delta$ satisfies $\psi
  \in \Phi$, $\delta\in V_\prec$, and $\psi\delta \not\in V_\prec$.
  Since the subgroup $\U$ is weakly leaf-invariant with respect to
  $\prec$, we have $\psi\delta \leadsto_\varphi \id$ which contradicts
  the assumption that $\sigma = \psi\delta \in V_<$. On the other
  hand, let $\sigma \in V_\prec$ be $\prec$-minimal so that $\sigma
  \not\in V_<$. As $\id \in V_<$, we have $\sigma \neq \id$ and hence, we can write $\sigma = \psi\delta$
  with $\psi \in \Phi$ and $\delta \in \Phi^*$. Since $V_\prec$ is a
  subtree of $\Phi^*$, we also have $\delta \in V_\prec$ and $\delta
  \prec \sigma$.  The minimality of $\sigma$ yields that $\delta \in
  V_<$. Since $\psi\delta \not\in V_<$, there exists $\gamma \in V_<$
  so that $\sigma = \psi\delta \leadsto_\varphi \gamma$.  Note that
  $\gamma \in V_< \subseteq V_\prec$ which contradicts that $\sigma =
  \psi\delta \in V_\prec$ as there would exists $\gamma \in V_\prec$
  so that $\sigma \leadsto_\varphi \gamma$ which is impossible.
\end{proof}
It can be shown that the subgroup ${\mc V} = \la
x_1,x_2,x_3,x_4\,x_1\,x_4^{-1},x_4^3\ra$ of the subgroup $\U$
in Section~\ref{sec:ExPre} is weakly leaf-invariant but it is not
leaf-invariant. The notion of a weakly leaf-invariant subgroup yields the following
\begin{lemma}\Label{lem:WeakInvIff}
  A normal subgroup $\UK\unlhd F$ is $\sigma$-invariant if and only if 
  $\sigma \leadsto_\varphi \id$. 
\end{lemma}
\begin{proof}
  Suppose that $\sigma \leadsto_\varphi \id$ holds. Then there exists
  a homomorphism $\pi\colon \im(\varphi) \to \im(\sigma\varphi)$
  so that $\sigma\varphi = \varphi\pi$. Let $g \in \UK = \Core_F(\UK) =
  \ker(\varphi)$ be given. Then $1 = (g^\varphi)^\pi = g^{\varphi\pi}
  = g^{\sigma\varphi} = (g^\sigma)^\varphi$ and so $g^\sigma
  \in \ker(\varphi) \subseteq \UK$. In particular, the subgroup $\UK$
  is $\sigma$-invariant. On the other hand, suppose that
  the normal subgroup $\UK\unlhd F$ is $\sigma$-invariant. For 
  $g \in F$, we define the map $\delta_g\colon \UK\bs F \to \UK\bs
  F,\: \UK\,t \mapsto \UK\,t\cdot g$. Note that, for $g,h\in F$, we have that
  $\delta_g \delta_h\colon \UK\bs F\to \UK\bs F,\: \UK\,t\mapsto
  \UK\,t\cdot gh$ and so $\delta_g \delta_h = \delta_{gh}$. 
  Then $\delta_g \in \im(\varphi)$. We define a map $\pi\colon \im(\varphi)
  \to \Sym(\UK\bs F),\: \delta_g \mapsto \delta_{g^\sigma}$. Let $g,h\in
  F$ be given. Then $(\delta_g \delta_h)^\pi = (\delta_{gh})^\pi =
  \delta_{(gh)^\sigma} = \delta_{g^\sigma\,h^\sigma} = \delta_{g^\sigma}
  \delta_{h^\sigma} = (\delta_g)^\pi (\delta_h)^\pi$. Suppose that, 
  for $g \in F$, the map $\delta_g$ acts trivially on $\UK\bs F$. Then,
  for each $t\in T$, we have $\UK\,t\cdot g = \UK\,t$ or $tgt^{-1} \in
  \UK$. Since $\UK\unlhd F$, the latter yields that $g \in \UK$ and, as
  $\UK$ is $\sigma$-invariant, we also have that $g^\sigma \in \UK$. Thus
  $tg^{\sigma}t^{-1} \in \UK$. Consider the image $(\delta_g)^\pi =
  \delta_{g^{\sigma}}$. Then, as $tg^\sigma t^{-1} \in \UK$, the map
  $\delta_g$ fixes $\UK\,t$. Because $t\in T$ was arbitrarily chosen, we
  have $\delta_{g^\sigma} = 1 \in \Sym(\UK\bs F)$. Thus the map $\pi$ defines
  a homomorphism that satisfies $\sigma\varphi = \varphi\pi$. Thus
  $\sigma\leadsto_\varphi \id$.
\end{proof}
Lemma~\ref{lem:WeakInvIff} yields that a $\Phi$-invariant normal subgroup is weakly
leaf-invariant. However, there exist subgroups which are weakly
leaf-invariant but not $\Phi$-invariant (e.g.  the subgroup $\U =
\la a,bab^{-1},b^3\ra$ of the Basilica group in Section~\ref{sec:ExPre}
satisfies $\sigma^2 \leadsto_\varphi \id$ but not $\sigma \leadsto_\varphi
\id$; thus, it is weakly leaf-invariant but not $\Phi$-invariant).
The condition $\UK\unlhd F$ in Lemma~\ref{lem:WeakInvIff} is necessary,
as we have the following
\begin{remark}\Label{rmk:Limits}
  The condition $\UK\unlhd F$ in Lemma~\ref{lem:WeakInvIff} is necessary,
  as the subgroup $U = \la a, b^2, ba^3b^{-1}, bab^{-2}a^{-1}b^{-1},
  ba^{-1}b^{-2}ab^{-1}\ra$ of the Basilica group $G$ is not normal in
  $G$, it satisfies $(\UK)^\sigma \subseteq \UK$; however, it does not satisfy $\sigma
  \leadsto_\varphi \id$.\smallskip

  On the other hand, the subgroup $U = \la a, bab, ba^{-1}b, b^4\ra$
  of the Basilica group $G$ satisfies $\sigma \leadsto_\varphi \id$ but
  it does not satisfy $(\UK)^\sigma \subseteq \UK$ as $[F:\Core_F(\UK)] =
  [F:\ti\El] = 8 \neq 4 = [F:\UK]$.
\end{remark}
A weakly leaf-invariant subgroup allows the following variant of our 
Reidemeister-Schreier theorem:
\begin{theorem}\Label{thm:InvCore}
  A weakly leaf-invariant normal subgroup which has finite index
  in an invariantly finitely $L$-presented group is invariantly finitely
  $L$-presented.
\end{theorem}
\begin{proof}
  Let $G=\la\X\mid\Q\mid\Phi\mid\R\ra$ be invariantly finitely
  $L$-presented and let $\U \cong \UK/K$ be a finite index normal
  subgroup of $G$.  As usual, we may consider $\Q = \emptyset$ as $G$
  is invariantly $L$-presented. Let $\ti V$ be the subset $\ti V \subseteq 
  V$ given by Lemma~\ref{lem:LeadstoSet}. Since $\U$ is weakly leaf-invariant, the
  weak-leafs $\Psi$ in Definition~\ref{def:WeakInv} satisfy
  $\Psi = \{ \psi\delta \mid \psi\in\Phi, \delta
  \in \ti V, \psi\delta\not\in \ti V\}$. By
  Lemma~\ref{lem:WeakInvIff}, each $\psi\delta \in \Psi$ induces an
  endomorphism of the normal subgroup $\UK\unlhd F$.
  Let $T$ be a Schreier transversal for $\UK$ in $F$ and let ${\mc Y}$ 
  denote the Schreier generators of the subgroup $\UK$. Then each 
  endomorphism $\psi\delta\in \Psi$ of $\UK$ translates to an endomorphism 
  $\widehat{\psi\delta}$ of the free group $F({\mc Y})$. Consider 
  the 
  invariant and finite $L$-presentation 
  \begin{equation}\Label{eqn:WeaklyInvSubgrpLp}
    \la {\mc Y} \mid \emptyset \mid 
    \{\widehat{\psi\delta}\mid \psi\delta\in \Psi\}\cup \{\widehat{\delta_t}\mid t\in T\} \mid 
    \{\tau(r^{\sigma}) \mid r\in \R, \sigma\in \ti V\}\ra,
  \end{equation}
  where $\delta_t$ denotes the endomorphism of $\UK$ which is induced
  by conjugation by $t\in T$. The finite $L$-presentation in
  Eq.~(\ref{eqn:WeaklyInvSubgrpLp}) defines the normal subgroup
  $\U$. This assertion follows with the same techniques as
  above; in particular, it follows from rewriting the presentation in
  Eq.~(\ref{eqn:SubgrpLPres}).
\end{proof}
The subgroup in Section~\ref{sec:ExPre} is a normal subgroup satisfying
$\sigma^2 \leadsto_\varphi \id$ and hence, Theorem~\ref{thm:InvCore} shows that
this subgroup is invariantly finitely $L$-presented. Even 
non-invariant $L$-presentations may give rise to invariant subgroup 
$L$-presentations as the following remark shows:
\begin{remark}
  There are non-invariant $L$-presentation $G =
  \la\X\mid\Q\mid\Phi\mid\R\ra$ and finite index subgroups $\U
  \leq G$ that satisfy $\UK ^ \sigma \subseteq \UK$ for each
  $\sigma\in\Phi^*$. For instance, the finite $L$-presentation of Baumslag's group $G$ in~\cite{Har08}
  is non-invariant (cf. Proposition~\ref{prop:NonInvLp}) while its index-$3$ subgroup $\U =
  \la a^3, b, t \ra$ satisfies $(\UK)^\sigma \subseteq \UK$
  for each $\sigma \in \Phi$. The subgroup $\U$ even admits 
  an invariant $L$-presentation over the generators $x = a^3$ and 
  $y = a^2ta^{-2}$ given by 
  \[
    \la \{x,y\} \mid \emptyset \mid \{\delta_t,\delta_{t^2}\} \mid \{y^{-1}xyx^{-4}\} \ra 
  \]
  where $\delta_t$ is induced by the mapping $x \mapsto x$ and $y\mapsto xyx^{-3}$ and 
  $\delta_{t^2}$ is induced by the mapping $x\mapsto x$ and $y\mapsto xyx^{-2}$. 
\end{remark}

The finite $L$-presentations for a finite index subgroup constructed
in Proposition~\ref{prop:ReidSchrInv}, Theorem~\ref{thm:LeafInv}, and
Theorem~\ref{thm:InvCore}, are derived from the group's $L$-presentation
$\la\X\mid\Q\mid\Phi\mid\R\ra$ by restricting to those endomorphisms in
$\Phi^*$ which restrict to the subgroup. However, there are subgroups
of an invariantly $L$-presented group that do not admit endomorphisms
in $\Phi^*$ which restrict to the subgroup. In this case the finite
$L$-presentation for the finite index subgroup needs to be constructed as
a finite extension of the finitely $L$-presented stabilizing core $\El$
as in the proof of Theorem~\ref{thm:CentralThm}. The following remark
gives an example of a subgroup of the invariantly finitely $L$-presented
Basilica group which does not admit endomorphisms in $\Phi^*$ that also
restrict to the subgroup:
\begin{remark}\Label{rem:LimOfProof}
  Let $\U = \la b^2, a^3, ab^2a^{-1}, a^{-1}b^2a, bab^{-1}a \ra$ denote
  a subgroup of the Basilica group $G$. Then $\U$ is a normal subgroup
  with index $6$ in $G$. We are not able to find an invariant finite
  $L$-presentation for $\U$.\smallskip
 
  The subgroup $\U$ admits the permutation representations $\varphi\colon 
  F \to \Sym(\UK\bs F)$ and the $\sigma$-iterates
  \[
    \varphi\colon \left\{ \begin{array}{rcl}
      a &\mapsto& (1,2,3)(4,6,5)\\
      b &\mapsto& (1,4)(2,5)(3,6)
    \end{array}\right. \quad\textrm{and}\quad
    \sigma\varphi\colon \left\{\begin{array}{rcl}
      a &\mapsto& (\:)\\
      b &\mapsto& (1,2,3)(4,6,5)
    \end{array}\right.
  \]
  as well as
  \[
    \sigma^2\varphi\colon \left\{\begin{array}{rcl}
      a &\mapsto& (1,3,2)(4,5,6) \\
      b &\mapsto& (\:)
    \end{array}\right. \quad\textrm{and}\quad
    \sigma^3\varphi\colon \left\{\begin{array}{rcl}
      a &\mapsto& (\:)\\
      b &\mapsto& (1,3,2)(4,5,6).
    \end{array}\right.
  \]
  Clearly, we have $\sigma^3 \leadsto_\varphi \sigma$ but for each 
  $0<\ell <3$ we do not have $\sigma^\ell \leadsto_\varphi \id$. Note 
  that the homomorphism $\pi \colon \im(\sigma\varphi) \to
  \im(\sigma^3\varphi)$ with $\sigma^3\varphi = \sigma\varphi\pi$ is
  bijective. Suppose there existed $\sigma^n \in \Phi^*$ so that the
  subgroup $\UK$ is $\sigma^n$-invariant. By Lemma~\ref{lem:WeakInvIff},
  the normal subgroup $\UK$ is $\sigma^n$-invariant if and only 
  if $\sigma^n \leadsto_\varphi \id$ holds. Clearly $n>3$. 
  Since $\sigma^n \leadsto_\varphi \id$ holds, there exists a homomorphism $\psi\colon\im(\varphi)
  \to \im(\sigma^n\varphi)$ so that $\sigma^n\varphi = \varphi\psi$.  We
  obtain $\varphi \psi = \sigma^n\varphi = \sigma^{n-3} \, \sigma^3\varphi
  = \sigma^{n-3}\,\sigma\varphi\pi = \sigma^{n-2}\varphi\pi$. Iterating
  this rewriting process eventually yields a positive integer $0\leq \ell < 3$ so
  that $\varphi\psi = \sigma^n \varphi = \sigma^\ell \varphi \pi^\ell$.
  As $\pi$ is bijective, this yields that $\sigma^\ell \varphi =\varphi
  \psi (\pi^{-1})^\ell$ and hence $\sigma^\ell \leadsto_\varphi \id$ which
  is a contradiction. Thus there is no positive integer $n\in \N$ so that 
  $\sigma^n \leadsto_\varphi \id$ and hence, no substitution in $\Phi^*$ 
  restricts to the subgroup $\UK$. 
\end{remark}
Our method to compute a finite $L$-presentation for the subgroup $\U$
in Remark~\ref{rem:LimOfProof} is therefore given by our explicit
proof of Theorem~\ref{thm:CentralThm}. If the subgroup $\U$ in
Remark~\ref{rem:LimOfProof} admits an invariant finite $L$-presentation,
then the substitutions may not be related to the substitutions $\Phi$
of the finite $L$-presentation $\la\X\mid\Q\mid\Phi\mid\R\ra$ of the
Basilica group in Proposition~\ref{prop:Basilica}. It is neither clear
to us whether $\U$ admits an invariant finite $L$-presentation nor do
we know how to possibly prove that $\U$ does not admit such invariant
$L$-presentation.

\section{Examples of subgroup {\boldmath$L$}-presentations}\Label{sec:Ex}
In this section, we again consider the subgroup $\U = \la a,bab^{-1},b^3\ra$
of the Basilica group $G$ as in Section~\ref{sec:ExPre}. We demonstrate
how our methods apply to this subgroup and, in particular, how to compute the 
$L$-presentation in Section~\ref{sec:ExPre}.\smallskip

Coset-enumeration for finitely $L$-presented groups~\cite{Har10b}
allows us to compute the permutation representation $\varphi\colon F
\to \Sym(\UK\bs F)$ for the group's action on the right-cosets. A Schreier transversal for 
$\U$ in $G$ is given by $T = \{ 1, b, b^2\}$ and we have 
\[
  \varphi\colon F \to \S_n,\: \left\{\begin{array}{rcl}
   a &\mapsto& (\:)\\
   b &\mapsto& (1,2,3).
  \end{array}\right.
\]
Moreover, $\U$ is a normal subgroup with index $3$ in $G$ and it satisfies $\sigma^2
\leadsto_\varphi \id$. By Lemma~\ref{lem:SimLeadsto}, there exists an integer $k \geq 2$ so that 
$\sigma^k \sim \id$; we can verify that $\sigma^4\varphi = \varphi$ holds and thus we have $\sigma^4 
\sim \id$. In particular, the subgroup $\U$ is (weakly) leaf-invariant and normal. Therefore 
the following techniques apply to this subgroup:
\begin{itemize}\addtolength{\itemsep}{-0.7ex}
\item As the subgroup $\U$ is a finite index subgroup of an invariantly
      $L$-presented group $G$, the general methods of Proposition~\ref{prop:ReidSchrInv} and
      Theorem~\ref{thm:NormInvSubgrps} apply.
\item As the subgroup $\U$ is leaf-invariant, the methods in
      Theorem~\ref{thm:LeafInv} apply.
\item As the subgroup $\U$ is weakly leaf-invariant and normal, the
      methods in Theorem~\ref{thm:InvCore} apply.
\end{itemize}
We demonstrate these different techniques for the subgroup $\U$. First, we
consider the general method from Proposition~\ref{prop:ReidSchrInv}. For
this purpose, we first note that the stabilizing subgroup $\El$ and
stabilizing core $\ti\El$ coincide by Corollary~\ref{cor:StabEq}. The
stabilizing subgroups $\El = \ti\El$ have index $9$ in $F$ and a Schreier
generating set for $\El = \ti\El$ is given by
\[
   \begin{array}{r@{~=~}l@{\quad}r@{~=~}l@{\quad}r@{~=~}l@{\quad}r@{~=~}l}
     x_{1} & a^3		&x_{4} & abab^{-1}a^{-2}	&x_{7} & a^2bab^{-1}		&x_{10} & b^2a^2ba^{-2}.\\
     x_{2} & bab^{-1}a^{-1}	&x_{5} & ab^2a^{-1}b^{-2}	&x_{8} & a^2b^2a^{-2}b^{-2}	&\multicolumn{2}{c}{}\\
     x_{3} & b^3		&x_{6} & b^2aba^{-1}		&x_{9} & b^2a^3b^{-2}		&\multicolumn{2}{c}{}
   \end{array}
\]
Let $F$ denote the free group over $\{a,b\}$ and let ${\mc F}$ denote
the free group over $\{x_1,\ldots,x_{10}\}$.  The Reidemeister rewriting
$\tau\colon F \to {\mc F}$ allows us to rewrite the iterated relation $r = [a, a^b]$.
We obtain $\tau( r ) = x_1^{-1}x_{10}^{-1}x_6\,x_{10}^{-1}x_9\,x_3$. Furthermore, the rewriting $\tau$
allows us to translate the substitution $\sigma$ of the Basilica group
to an endomorphism of the free group ${\mc F}$. For instance, we obtain 
a free group homomorphisms $\widehat\sigma\colon {\mc F} \to {\mc F}$ which is induced by the map
\[
   \begin{array}{rcl@{\qquad}rcl}
     x_{1} &\mapsto& x_3^2,			&x_{6} &\mapsto& x_8\,x_9,		\\
     x_{2} &\mapsto& x_5,			&x_{7} &\mapsto& x_3\,x_2\,x_5\,x_6,		\\
     x_{3} &\mapsto& x_1, 			&x_{8} &\mapsto& x_3\,x_2\,x_4\,x_{10}^{-1}x_8^{-1},\\
     x_{4} &\mapsto& x_6\,x_2^{-1}x_3^{-1},	&x_{9} &\mapsto& x_8\,x_{10}\,x_8\,x_{10}	\\
     x_{5} &\mapsto& x_8^{-1},			&x_{10} &\mapsto& x_8\,x_{10}\,x_7\,x_3^{-1}.
   \end{array}
\]
Similarly, the conjugation actions $\delta_a$ and $\delta_b$ which are induced by $a$
and $b$, respectively, translate to endomorphisms $\widehat\delta_a$ and $\widehat\delta_b$ of the free group ${\mc
F}$. By Proposition~\ref{prop:ReidSchrInv}, the stabilizing subgroups $\El
= \ti\El$ are finitely $L$-presented by
\[
  L = \El/K \cong \la \{x_1,\ldots,x_{10}\} \mid \emptyset \mid \{\widehat\sigma,\widehat\delta_a,\widehat\delta_b\}
  \mid \{x_1^{-1}x_{10}^{-1}x_6\,x_{10}^{-1}x_9\,x_3\} \ra.
\]
The subgroup $\U$ satisfies the short exact sequence $1 \to
L \to \U \to C_3 \to 1$ with a cyclic group $C_3 = \la \alpha \mid
\alpha^3 = 1 \ra$ of order $3$. Corollary~\ref{cor:FinExt} yields the 
following finite $L$-presentation for the subgroup $\U$:
\[
  \la \{\alpha,x_1,\ldots,x_{10}\} \mid \{ \alpha^3 x_1^{-1} \} \cup \{ (x_i^{-1})^\alpha x_i^{\delta_a} \}_{1\leq i\leq 10} 
  \mid \ti\Psi \mid \{x_1^{-1}x_{10}^{-1}x_6\,x_{10}^{-1}x_9\,x_3\} \ra.
\]
where the substitutions
$\widehat\Psi = \{\widehat\sigma,\widehat\delta_a,\widehat\delta_b\}$ of $L$'s finite $L$-presentation
are dilated to endomorphisms $\ti\Psi =
\{\widetilde\sigma,\widetilde\delta_a,\widetilde\delta_b\}$ of
the free group over $\{\alpha,x_1,\ldots,x_{10}\}$ as in the proof of
Proposition~\ref{prop:FpExt}.\smallskip

Secondly, the subgroup $\U$ is (weakly) leaf-invariant and
normal and therefore, the methods in Section~\ref{sec:InvSubgrpLpres} apply.
First, we consider the construction in Theorem~\ref{thm:LeafInv} for leaf-invariant subgroups. A
Schreier generating set for the subgroup $\UK$ is given by $x_1 = a$, $x_2
= bab^{-1}$, $x_3 = b^2ab^{-2}$, and $x_4 = b^3$.  Since $\sigma^4\varphi
= \varphi$ holds, the subgroup $\U$ is $\sigma^4$-invariant and its suffices to 
rewrite the relation $r = [a,b]$ and its images. The images $\tau(tr^{\sigma^i}t^{-1})$ have the form:
\[
  {\small
  \begin{array}{cccc}
    \toprule
    i & t = 1 & t = b & t = b^2 \\
    \midrule
    0 & x_1^{-1}x_4^{-1}x_3^{-1}x_4\,x_1\,x_4^{-1}x_3\,x_4& x_2^{-1}x_1^{-1}x_2\,x_1 & x_3^{-1}x_2^{-1}x_3\,x_2\\
    1 & x_4^{-1}x_2^{-1}x_4^{-1}x_3\,x_4\,x_2^{-1}x_4\,x_1 & x_4^{-1}x_3^{-1}x_1\,x_3^{-1}x_4\,x_2 &
        x_1^{-1}x_4^{-1}x_2\,x_4\,x_1^{-1}x_3\\
    2 & x_1^{-2}x_4^{-1}x_2^{-2}x_4\,x_1^2\,x_4^{-1}x_2^2\,x_4 & x_2^{-2}x_4^{-1}x_3^{-2}x_4\,x_2^2\,x_4^{-1}x_3^2\,x_4 &
        x_3^{-2}x_1^{-2}x_3^2\,x_1^2 \\
    3 & x_4^{-2}x_3^{-2}x_4^{-1}x_2^2\,x_4\,x_3^{-2}x_4^2\,x_1^2 & x_4^{-1}x_1^{-2}x_4^{-2}x_3^2\,x_4^2\,x_1^{-2}x_4\,x_2^2 &
        x_4^{-1}x_2^{-2}x_4^{-1}x_1^2\,x_4\,x_2^{-2}x_4\,x_3^2\\
    \bottomrule
  \end{array}}
\]
Let $\R$ denote the set of relations above. The endomorphism $\sigma^4$ translates, via $\tau$, to
an endomorphism of the free group over $\{x_1,\ldots,x_4\}$ which is induced by 
\[
  \widehat{\sigma^4}\colon \left\{\begin{array}{rcl}
    x_1 &\mapsto& x_1^4\\
    x_2 &\mapsto& x_4\,x_2^4\,x_4^{-1}\\
    x_3 &\mapsto& x_4^2\,x_3^4\,x_4^{-2}\\
    x_4 &\mapsto& x_4^4.
  \end{array}\right.
\]
By Theorem~\ref{thm:LeafInv}, an $L$-presentation for the subgroup $\U$ is given by 
\[
  \U \cong \la \{x_1,\ldots,x_4\} \mid \emptyset \mid \{\widehat{\sigma^4}\} \mid \R \ra.
\]
Finally, the subgroup $\U$ is weakly leaf-invariant and normal and therefore, the methods in 
Theorem~\ref{thm:InvCore} apply.
As $\sigma^2 \leadsto_\varphi \id$ holds, it suffices to consider the relations $\tau(r)$ and 
$\tau(r^\sigma)$ and their images under the substitutions $\widehat{\sigma^2}$ and 
$\widehat{\delta_b}$ (as a Schreier transversal is given by $T = \{1,b,b^2\}$) which 
are induced by 
\[
  \widehat{\sigma^2}\colon \left\{\begin{array}{rcl}
   x_1 &\mapsto& x_1^2 \\
   x_2 &\mapsto& x_3^2 \\
   x_3 &\mapsto& x_4\,x_2^2\,x_4^{-1} \\
   x_4 &\mapsto& x_4^2
  \end{array}\right.\quad\textrm{and}\quad
  \widehat{\delta_b}\colon \left\{\begin{array}{rcl}
   x_1 &\mapsto& x_2 \\
   x_2 &\mapsto& x_3 \\
   x_3 &\mapsto& x_4\,x_1\,x_4^{-1}\\ 
   x_4 &\mapsto& x_4.
  \end{array}\right.
\]
Theorem~\ref{thm:InvCore} yields the following finite $L$-presentation for the subgroup $\U$:
\[
  \U \cong \la \{x_1,\ldots,x_4\} \mid \emptyset \mid \{\widehat{\sigma^2},
  \widehat{\delta_b}\} \mid \{\tau(r),\tau(r^\sigma)\} \ra.
\]

\subsection{An application to the Grigorchuk group}
As a finite $L$-presentation of a group allows the application
of computer algorithms, we may use our constructive proof of
Theorem~\ref{thm:CentralThm} allows us to investigate the structure of
a self-similar group by its finite index subgroups as in~\cite{HR94}.
As an application, we consider the Grigorchuk group, see~\cite{Gri80},
$\Grig = \la a,b,c,d\ra$ and its normal subgroup $\la d\ra^G$. We show
that the subgroup $\la d\ra^G$ has a minimal generating set with $8$
elements and thereby we correct a mistake in~\cite{BG02,Gr05}.\smallskip

The Grigorchuk group $\Grig$ satisfies the following well-known 
\def\0{\cite{Lys85}}
\begin{proposition}[Lys\"enok,~\0]\Label{thm:Grig}
  The Grigorchuk group $\Grig$ is invariantly\linebreak $L$-pre\-sented by
  $\Grig\cong \left\la \{a,b,c,d\} \mid \{a^2,b^2,c^2,d^2,bcd\} \mid \{\sigma\}\mid 
        \{ (ad)^4,(adacac)^4 \} \right\ra$,\linebreak
  where $\sigma$ is the endomorphism of the free group over 
  $\{a,b,c,d\}$ induced by the mapping $a\mapsto aca$, $b\mapsto d$,
  $c\mapsto b$, and $d\mapsto c$.
\end{proposition}
It was claimed in~\cite[Section~4.2]{BG02} and in~\cite[Section~6]{Gr05}
that the normal closure $\la d\ra^\Grig$ is $4$-generated by
$\{d,d^a,d^{ac},d^{aca}\}$. In the following, we show that the
Reidemeister Schreier Theorem can be used to proof that a generating set
for $\la d\ra^\Grig$ contains as least $8$ elements. Our coset-enumeration
for finitely $L$-presented groups~\cite{Har10b} and our solution to the
subgroup membership problem for finite index subgroups in~\cite{Har10b}
show that the subgroup
\begin{equation}\Label{eqn:NormDGens}
  {\mc D} = \la\,d,d^a,d^{ac},d^{aca},d^{acac},d^{acaca},d^{acacac},d^{acacaca}\,\ra
\end{equation} 
has index $16$ in $\Grig$ and it is a normal subgroup of $\Grig$ so that 
$\Grig / {\mc D}$ is a dihedral group of order $16$. In particular, the subgroup ${\mc
D}$ and the normal closure $\la d \ra^\Grig$ coincide. A
permutation representation $\varphi\colon F \to \S_n$ for the group's
action on the right-cosets $\UK\bs F$ is given by
\[
  \varphi\colon F \to \S_{16},\:\left\{\begin{array}{rcl}
    a &\mapsto& (1,2)(3,5)(4,6)(7,9)(8,10)(11,13)(12,14)(15,16) \\
    b &\mapsto& (1,3)(2,4)(5,7)(6,8)(9,11)(10,12)(13,15)(14,16) \\ 
    c &\mapsto& (1,3)(2,4)(5,7)(6,8)(9,11)(10,12)(13,15)(14,16) \\
    d &\mapsto& (\:).
  \end{array}\right.
\]
Our variant of the Reide\-meister-Schreier Theorem 
and the techniques introduced in
Section~\ref{sec:InvSubgrpLpres} enable us to compute a subgroup
$L$-presentation for ${\mc D}$. For this purpose, we first note
that $\sigma^3 \leadsto_\varphi \id$ holds and hence, the normal core ${\mc D} =
\Core_F(\UK) = \ker(\varphi)$ is $\sigma^3$-invariant. The normal core
$\Core_F(\UK)$ has rank $49$ and a Schreier transversal for ${\mc D}$
in $\Grig$ is given by
\[
  1, a, b, ab, ba, aba, bab, (ab)^2, (ba)^2, a(ba)^2,
  b(ab)^2, (ab)^3, (ba)^3, 
  a(ba)^3, b(ab)^3, (ab)^4.
\]
A finite $L$-presentation with generators $d_0 = d$, $d_1 = d^a$,
$d_2 = d^{ac}$, $d_3 = d ^{aca}$, $d_4 = d^{acac}$, $d_5 = d^{acaca}$,
$d_6 = d^{acacac}$, and $d_7 = d^{acacaca}$ is given by 
\[
   {\mc D} \cong \la\{d_0,\ldots,d_7\}\mid \emptyset \mid
   \{\widehat\sigma, \delta_a,\delta_b\} \mid \R\,\ra,
\]
where the iterated relations are
\begin{eqnarray*}
  \R = \left\{
  d_0^2, [d_1,d_0], [d_1,d_4],
  \left[d_7,d_3\,d_4\right]^4,
  [d_7\,d_0,d_3\,d_4],
  (d_3\,d_7\,d_4\,d_0)^2, 
  (d_7\,d_4^{d_3}\,d_0\,d_3^{d_4})^2
  \right\}
\end{eqnarray*}
and the endomorphisms are induced by the maps
\[
  \delta_a\colon \left\{ \begin{array}{rcl}
   d_0 &\mapsto& d_1\\
   d_1 &\mapsto& d_0\\
   d_2 &\mapsto& d_3\\
   d_3 &\mapsto& d_2\\
   d_4 &\mapsto& d_5\\
   d_5 &\mapsto& d_4\\
   d_6 &\mapsto& d_7\\
   d_7 &\mapsto& d_6
  \end{array}\right.,\quad
  \delta_b\colon \left\{ \begin{array}{rcl}
   d_0 &\mapsto& d_0\\[0.75ex]
   d_1 &\mapsto& d_2\\[0.75ex] 
   d_2 &\mapsto& d_1\\[0.75ex] 
   d_3 &\mapsto& d_4^{d_0}\\[0.75ex]
   d_4 &\mapsto& d_3^{d_0}\\[0.75ex]
   d_5 &\mapsto& d_6\\[0.75ex] 
   d_6 &\mapsto& d_5\\[0.75ex] 
   d_7 &\mapsto& d_7^{d_0}
  \end{array}\right.,\quad\textrm{and}\quad
  \widehat\sigma\colon \left\{ \begin{array}{rcl}
   d_0 &\mapsto& d_0\\
   d_1 &\mapsto& d_0 ^ {\,d_7^{d_3}}\\ 
   d_2 &\mapsto& d_0 ^ {\,d_7^{d_4}}\\ 
   d_3 &\mapsto& d_0 ^ {\,d_7^{d_4}d_7^{d_3}}\\ 
   d_4 &\mapsto& d_0 ^ {\,d_7^{d_3}d_7^{d_4}}\\ 
   d_5 &\mapsto& d_0 ^ {\,d_7^{d_3}d_7^{d_4}d_7^{d_3}}\\ 
   d_6 &\mapsto& d_0 ^ {\,d_7^{d_4}d_7^{d_3}d_7^{d_4}}\\ 
   d_7 &\mapsto& d_0 ^ {\,d_7^{d_4}d_7^{d_3}d_7^{d_4}d_7^{d_3}}. 
  \end{array}\right.
\]
The $L$-presentation of ${\mc D}$ allows us to compute the abelianization
${\mc D} / [{\mc D},{\mc D}]$ with the methods from~\cite{BEH08}.
We obtain that ${\mc D} / [{\mc D},{\mc D}] \cong (\Z/2\Z)^8$ is
$2$-elementary abelian of rank $8$. Hence, the normal subgroup ${\mc D}$
has a minimal generating set of length $8$ because a generating set with
$8$ generators was already given in Eq.~(\ref{eqn:NormDGens}) above.
In particular, this shows that $\la d\ra^\Grig \neq \la d, d^a, d^{ac},
d^{aca}\ra$. Note that the mistake could have been detected by computing
the abelianization of the image of $\la d\ra^\Grig$ in a finite quotient
of $\Grig$ (e.g. the quotient $G / \Stab(n)$ for $n\geq 4$), by hand or
using a computer-algebra-system such as \Gap.

\subsection*{Acknowledgments}
I am grateful to Laurent Bartholdi for valuable comments and suggestions.


\def\cprime{$'$}

\noindent Ren\'e Hartung,
{\scshape Mathematisches Institut},
{\scshape Georg-August Universit\"at zu G\"ottingen},
{\scshape Bunsenstra\ss e 3--5},
{\scshape 37073 G\"ottingen}
{\scshape Germany}\\[1ex]
{\it Email:} \qquad \verb|rhartung@uni-math.gwdg.de|
\end{document}